\documentclass[11pt]{article}
\usepackage{amsmath, amsthm, amssymb,amscd, amsfonts}

\usepackage{tkz-euclide, float}
\usetikzlibrary{shapes, calc} 
\usetikzlibrary{bbox}
\usetikzlibrary{decorations.markings, calligraphy}
\usetikzlibrary{decorations.pathmorphing}

\usepackage{enumerate, enumitem, bbm, bm}

\usepackage{fullpage, verbatim, subcaption, array, tabularx, comment}
\usepackage{tikz, xcolor, graphicx}
\tikzstyle{uStyle}=[shape = circle, minimum size = 2pt, inner sep =2.5pt, outer sep = 0pt, draw, fill=white]
\tikzstyle{myStyle}=[shape = circle, draw, fill=black, scale=0.5]

\usepackage{url,hyperref,nameref}
\hypersetup{colorlinks=true,linkcolor=black,anchorcolor=blue,citecolor=black,filecolor=blue,urlcolor=blue,bookmarksnumbered=true,pdfview=FitB}

\usepackage{thmtools}
\usepackage{thm-restate}
\usepackage{soul}
\usepackage{cleveref}
\usepackage[all,arc]{xy}
\usepackage{mathrsfs}
\usepackage{mathtools}

\usepackage{todonotes}

\RequirePackage{marginnote,hyperref}
\newtheorem{lem}{Lemma}

\newtheorem{thm}{Theorem}
\newtheorem{claim}{Claim}
\newtheorem{obs}{Observation}
\newtheorem*{mainthm}{Main Theorem}
\newtheorem*{key-lem}{Key Lemma}
\newtheorem{conj}{Conjecture}

\theoremstyle{definition}
\newtheorem{definition}{Definition}

\newcommand{\vph}{\varphi}
\newcommand{\HH}{\mathcal{H}}
\newcommand{\A}{\mathcal{A}}
\newcommand{\B}{\mathcal{B}}
\newcommand{\C}{\mathcal{C}}
\newcommand{\D}{\mathcal{D}}
\newcommand{\T}{\mathcal{T}}
\newcommand{\U}{\mathcal{U}}
\newcommand{\W}{\mathcal{W}}
\newcommand{\I}{\mathcal{I}}
\newcommand{\G}{\mathcal{G}}
\newcommand{\mad}{\textrm{mad}}

\title{Equitable Coloring in 1-Planar Graphs}
\date{\today}

\author{Daniel W. Cranston\thanks{Department of Computer Science, Virginia Commonwealth University, Richmond, VA, USA; \texttt{dcranston@vcu.edu}} \and Reem Mahmoud\thanks{Department of Computer Science, Virginia Commonwealth University, Richmond, VA, USA; \texttt{mahmoudr@vcu.edu}}}

\begin{document}

\maketitle

\begin{abstract}
For every $r\ge13$, we show every 1-planar graph $G$ with $\Delta(G)\le r$ has an equitable $r$-coloring.
\end{abstract}

Let $[r]:=\{1,\dots,r\}$. A (proper) \emph{$r$-coloring} is a map $\vph:V(G)\to[r]$ such that $\vph(u)\neq\vph(v)$ for every $uv\in E(G)$. An $r$-coloring $\vph$ is \emph{equitable} if $|\vph^{-1}(i)|-|\vph^{-1}(j)|\le1$ for every $i,j\in[r]$; that is, color classes differ in size by at most 1. Equitable coloring was first introduced by Meyer in the seventies \cite{Meyer} motivated by its applications in scheduling problems. With many scheduling problems, we are often interested in assigning time slots to numerous tasks efficiently. That usually means assigning different time slots to tasks which cannot be performed simultaneously (finding a proper coloring), and further aiming for a nearly equal number of tasks performed in each time slot (finding an equitable coloring). 

One example mentioned by Meyer is assigning time slots to garbage collection routes so that routes which cannot be run on the same day receive different time slots, and the number of routes assigned to each time slot is as balanced as possible. Another example is assigning time slots to college classes so that no two college classes with common students are given the same time slot, and the classes are spread out as evenly as possible amongst the different time slots. Both examples can be translated into equitable coloring problems. For instance, the second problem is equivalent to equitably coloring the graph whose vertices are the college classes, where two vertices are adjacent if their corresponding classes share a student. 

As a result, equitable coloring has been studied extensively; 
see \cite{Lih-survey} for a survey. In particular, a lot of work has focused on studying the following two invariants. The \emph{equitable chromatic number}, $\chi_e(G)$, of a graph $G$ is the smallest integer $r$ such that $G$ admits an equitable $r$-coloring. The \emph{equitable chromatic threshold}, $\chi_e^*(G)$, of a graph $G$ is the smallest integer $r$ such that $G$ admits an equitable $k$-coloring for every $k\ge r$. The gap between $\chi_e(G)$ and $\chi_e^*(G)$ can be arbitrarily large. For example, $\chi_e(K_{7,7})=2$ while $\chi_e^*(K_{7,7})=8$ (in fact, $K_{7,7}$ admits equitable $k$-colorings for $k\in\{2,4,6\}$ and $k\ge8$, but not for $k\in\{3,5,7\}$). Meyer conjectured the following about the equitable chromatic number.

\begin{conj}[Equitable Coloring Conjecture]
\label{ECC}
If $G$ is connected, but not a complete graph or an odd cycle, then $G$ has an equitable $\Delta(G)$-coloring.
\end{conj}

This conjecture, which mirrors Brooks' Theorem, was proposed a few years after Hajnal and Szemer\'{e}di \cite{Hajnal-Szemeredi} had already proved the following about the equitable chromatic threshold.

\begin{thm}[Hajnal-Szemer\'{e}di]
\label{HS-thm}
Every graph $G$ has an equitable $r$-coloring for each $r\ge (\Delta(G)+1)$.
\end{thm}

The bound in Theorem~\ref{HS-thm} is sharp for both complete graphs $K_{\Delta+1}$ and odd cycles, since they admit no $\Delta$-colorings, and also sharp for complete bipartite graphs $K_{\Delta,\Delta}$ with $\Delta$ odd, since they admit no equitable $\Delta$-colorings. Chen, Lih, and Wu \cite{Chen-Lih-Wu} conjectured a stronger chromatic-threshold version of Meyer's conjecture.

\begin{conj}[Equitable $\Delta$-Coloring Conjecture]
\label{EDCC}
If $G$ is connected, but not a complete graph, a complete bipartite graph, or an odd cycle, then $G$ has an equitable $r$-coloring for every $r\ge\Delta(G)$. 
\end{conj}

The Equitable $\Delta$-Coloring Conjecture has been proved for various classes of graphs. Lih and Wu \cite{Lih-Wu} proved it for bipartite graphs. Kostochka and Nakprasit \cite{Kostochka-Nakprasit1} proved it for graphs $G$ with $\mad(G)\le\frac{\Delta(G)}{5}$ and $\Delta(G)\ge46$.
They also proved it \cite{Kostochka-Nakprasit2} for $d$-degenerate graphs $G$ with $d\le\frac{\Delta(G)-1}{14}$. Tian and Zhang \cite{Tian-Zhang} proved it for pseudo-outerplanar graphs. 
Yap and Zhang \cite{Yap-Zhang1} proved it for outerplanar graphs and \cite{Yap-Zhang2} for planar graphs $G$ with $\Delta(G)\ge13$. Nakprasit \cite{Nakprasit} improved the hypothesis of Yap and Zhang for planar graphs $G$ to $\Delta(G)\ge9$. More recently, Kostochka, Lin, and Xiang \cite{Kostochka-Lin-Xiang} improved the hypothesis again to $\Delta(G)\ge8$. In the same paper, they extended Nakprasit's result for graphs $G$ with $\Delta(G)\ge9$ from only planar graph to include all such graphs embeddable into any surface with nonnegative Euler characteristic. Finally, Zhang \cite{Zhang} proved the conjecture for 1-planar graphs $G$ with $\Delta(G)\ge17$ (in fact, he proved it for a wider class of graphs $\G$ where $|E(H)|\le 4|V(H)|-8$ for every subgraph $H$ of a graph $G\in\G$). Using techniques similar to those in \cite{Kostochka-Lin-Xiang}, we strengthen Zhang's result as follows.

\begin{mainthm}
\label{main-thm}
If $r\ge13$ and $G$ is a 1-planar graph with $\Delta(G)\le r$, then $G$ has an equitable $r$-coloring.
\end{mainthm}

The only place where we use 1-planarity is to establish the edge bounds in Lemma~\ref{edges-lem} below. Hence, the \hyperref[main-thm]{Main Theorem} is true, more generally, for every graph $G$ such that $G$ and all of its subgraphs satisfy the edge bounds in Lemma~\ref{edges-lem}. Our result does not immediately extend to the wider class of graphs $\G$ considered by Zhang (defined above) since there is no clear edge bound for the class of bipartite graphs in $\G$. However, it is only in Claims~\ref{A-B-edges} and \ref{strong-comp} that our proof uses the edge bound from Lemma~\ref{edges-lem} for bipartite 1-planar graphs. This suggests that it may be possible to adapt our proof to work for all graphs in $\G$.


\section{Preliminaries}\label{prelims} 

Let \emph{$\Delta(G)$} denote the maximum degree of a graph $G$. Denote by \emph{$E(A,B)$} the set of edges between disjoint vertex subsets $A$ and $B$. A \emph{$v,w$-path} is a path from $v$ to $w$. A graph $G$ is \emph{$d$-degenerate} if every subgraph $H$ of $G$ contains a vertex of degree at most $d$; equivalently, there exists a \emph{degeneracy ordering} $\sigma$ of $V(G)$ such that every $v\in\sigma$ has at most $d$ neighbors later in $\sigma$. A \emph{directed graph (digraph)} $\D$ is a graph with directed edges called \emph{arcs}. Note that digraphs may contain parallel edges of opposite directions; each pair of such edges is a \emph{digon}. For an arc $vw$ in a digraph, $v$ is an \emph{in-neighbor} of $w$ and $w$ is an \emph{out-neighbor} $v$. A \emph{sink} is a non-isolated vertex whose neighbors are all in-neighbors. A component in a digraph is \emph{strong} or \emph{strongly connected} if there is a directed path from $v$ to $w$ for every $v,w\in V(\D)$. A subgraph $H$ of a (di)graph $G$ is \emph{spanning} if $V(H)=V(G)$. A \emph{tree} is a connected acyclic graph. An \emph{in-tree} $\I$ is a directed tree with a sink $s$ such that there is a directed $v,s$-path for every $v\in V(\I)$. A graph $G$ is \emph{1-planar} if it can be drawn in the plane such that each edge crosses at most one other edge.

\begin{lem}[\cite{Karpov}]
\label{edges-lem}
If $G$ is a 1-planar graph, then $|E(G)|\le 4|V(G)|-8$. Moreover, if $G$ is a bipartite 1-planar graph with $|V(G)|\ge 3$, then $|E(G)|\le 3|V(G)|-6$.
\end{lem}

\begin{proof}
Let $G$ be a 1-planar graph (not necessarily bipartite). We first show that the number $c(G)$ of crossings in $G$ (i.e. points at which edges cross) is at most $|V(G)|-2$. To see this, create an $(|V(G)|+c(G))$-vertex graph $G'$ from $G$ by replacing crossings in $G$ with vertices. Since no two crossings in $G$ are adjacent, it is possible to form a maximal planar graph $G''$ from $G'$ by adding edges between vertices of $G$.\footnote{To be precise, we allow $G''$ to have parallel edges, but no 2-faces.  Fortunately, Euler's formula still holds for all such planar multigraphs.} By Euler's formula, the number of faces in $G''$ is $2|V(G'')|-4=2(|V(G)|+c(G))-4$. Since every crossing of $G$ is incident to four 3-faces in $G''$ and no two crossings share a face, $2(|V(G)|+c(G))-4\ge 4c(G)$ which implies $c(G)\le |V(G)|-2$.

Now we delete an edge from each crossing of $G$. By Euler's formula, this gives a planar graph $H$ with at most $3|V(G)|-6$ edges. So $|E(G)|=|E(H)|+c(G)\le 3|V(G)|-6+|V(G)|-2$. This proves the first statement. 

If $G$ is also bipartite, then $H$ is a bipartite planar graph; so by Euler's formula $H$ has at most $2|V(G)|-4$ edges. Thus, $|E(G)|\le 2|V(G)|-4+|V(G)|-2$. This proves the second statement.
\end{proof}

\begin{lem}
\label{mindeg-lem}
If $G$ is a 1-planar graph, then $G$ is 7-degenerate. In particular, $\delta(G)\le7$. 
\end{lem}

\begin{proof}
The second statement follows from the first, by the definition of degeneracy. Let $H$ be a subgraph of $G$. Lemma~\ref{edges-lem} gives $|E(H)|\le 4|V(H)|-8$. So the average vertex degree in $H$ is $\frac{\sum_{v\in V(H)}d(v)}{|V(H)|}=\frac{2|E(H)|}{|V(H)|}\le \frac{2(4|V(H)|-8)}{|V(H)|}<8$. Thus, $H$ has some vertex with degree at most 7.
\end{proof}

The first bound in Lemma~\ref{edges-lem} is sharp.\footnote{Regarding the second bound, when $|V(G)|\ge 4$ Karpov~\cite{Karpov} proved the stronger form $|E(G)|\le 3|V(G)|-8$, which is sharp infinitely often.} For example, consider an embedding of the 3-cube $Q_3$ as two nested squares whose corresponding vertices are adjacent. Adding both diagonals to each 4-face of $Q_3$ gives a 1-planar graph with 8 vertices and $24=4(8)-8$ edges. Moreover, it is possible to construct 1-planar graphs $G$ with $\delta(G)=7$ by adding diagonals to 4-faces in 4-regular planar graphs that only contain 3-faces and 4-faces\footnote{Consider the  rhombicuboctahedron graph which has eight 3-faces and eighteen 4-faces. Each vertex is incident to three 4-faces and one 3-face and has degree 4. Adding both diagonals to every 4-face gives a 7-regular 1-planar graph.  In fact, this construction can be extended to give an infinite family of sharpness examples.}. So, the bound in Lemma~\ref{mindeg-lem} is also sharp. 

The following lemma implies that it suffices to prove the \hyperref[main-thm]{Main Theorem} for graphs whose order is divisible by the number of colors. Such graphs are easier to work with since their equitable colorings have color classes of equal size.

\begin{lem}
\label{divisibility-lem}
If the \hyperref[main-thm]{Main Theorem} is true for every 1-planar graph $G$ with $|V(G)|$ divisible by $r$, then it is true for every 1-planar graph. 
\end{lem}

\begin{proof}
Assume the \hyperref[main-thm]{Main Theorem} is true for every 1-planar graph whose order is divisible by $r$. Let $G$ be a 1-planar graph with $|V(G)|=rs-t$ for some $s\ge1$ and for $0<t<r$.

If $t\le 6$, then let $G'$ be the disjoint union of $G$ and $K_t$ (it is straightforward to verify that $K_t$ is 1-planar). Now $G'$ is 1-planar and $|V(G')|=rs$. By assumption, $G'$ admits an equitable $r$-coloring $\vph'$ where each color class has size $s$. Since all vertices in this $K_t$ receives distinct colors, the restriction $\vph$ of $\vph'$ to $G$ is an equitable $r$-coloring of $G$.

If instead $t\ge7$, then recall that $G$ is 7-degenerate, so has a degeneracy ordering $\sigma$, by Lemma~\ref{mindeg-lem}. Let $\sigma:=v_1,v_2,\dots,v_{|V(G)|}$, $R:=\{v_1,v_2,\dots,v_{r-t}\}$, $G':=G\setminus R$, and $\sigma'$ be the restriction of $\sigma$ to $G'$. Now $G'$ is 1-planar and $|V(G')|=r(s-1)$. By assumption, $G'$ admits an equitable $r$-coloring $\vph'$ where each class has size $s-1$. We color the vertices of $R$ in the reverse order starting from $v_{r-t}$ and ending with $v_1$. Observe that each $v_i$ in $R$ has at most 7 neighbors later in $\sigma'$. Further, at the time of coloring $v_i$, the number of colored vertices in $R$ is $r-t-i\le r-7-1=r-8$ since $i\ge1$ and $t\ge7$. Thus, there exists a color for $v_i$ distinct from the colors of its neighbors in $\sigma'$ and the colored vertices in $R$. Thus, $G$ has an $r$-equitable coloring. 
\end{proof}

As mentioned in the introduction, we adopt many of the ideas used in the recent paper of Kostochka, Lin, and Xiang \cite{Kostochka-Lin-Xiang}. We first recall some of the definitions and concepts from that paper along with some new ones. Following Definition~\ref{digraph-def}, we will give a brief sketch of the proof of the \hyperref[main-thm]{Main Theorem}, and also  provide intuition for some terms in this definition.

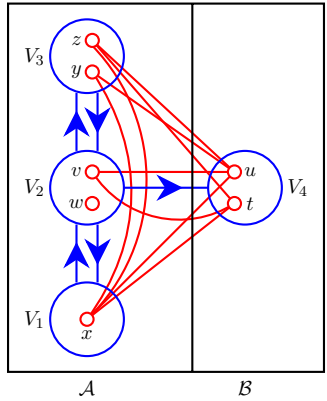
\begin{figure}[!h]
\centering
\begin{tikzpicture}[scale=0.7,every node/.style={scale=0.7}]
\begin{scope}
\draw[thick] (0.1,1.3) edge[red] (2.8,-1.2) (0.1,1.3) edge[red] (2.8,-1.8) (0.1,-1.2) edge[red] (2.8,-1.2) (0.1,-1.2) edge[red, bend right=45] (2.8,-1.8) (0,-4) edge[red] (2.8,-1.2) (0,-4) edge[red] (2.8,-1.8) (2.8,-1.2) edge[red] (0.1,0.7) (0,-4) edge[red, bend right=35] (0.1,0.7) (0,-4) edge[red, bend right=45] (0.1,1.2);
\draw[thick] (0.2,0.3) edge[blue, decoration={markings, mark=at position 0.7 with {\arrow{Stealth[length=10pt,width=10pt]}}},postaction={decorate}] (0.2,-0.8) (-0.2,-0.8) edge[blue, decoration={markings, mark=at position 0.7 with {\arrow{Stealth[length=10pt,width=10pt]}}},postaction={decorate}] (-0.2,0.3) (0.2,-2.2) edge[blue, decoration={markings, mark=at position 0.7 with {\arrow{Stealth[length=10pt,width=10pt]}}},postaction={decorate}] (0.2,-3.3) (-0.2,-3.3) edge[blue, decoration={markings, mark=at position 0.7 with {\arrow{Stealth[length=10pt,width=10pt]}}},postaction={decorate}] (-0.2,-2.2) (0.7,-1.5) edge[blue, decoration={markings, mark=at position 0.7 with {\arrow{Stealth[length=10pt,width=10pt]}}},postaction={decorate}] (2.3,-1.5);
\draw[thick] (0,-4) node[uStyle, draw=red] {} (0,-4.3) node {$x$} (0.1,-1.8) node[uStyle, draw=red] {} (0.1,-1.2) node[uStyle, draw=red] {} (-0.2,-1.2) node {$v$} (-0.2,-1.8) node {$w$} (0.1,0.7) node[uStyle, draw=red] {} (0.1,1.3) node[uStyle, draw=red] {} (-0.2,0.7) node {$y$} (-0.2,1.3) node {$z$} (2.8,-1.8) node[uStyle, draw=red] {} (2.8,-1.2) node[uStyle, draw=red] {} (3.1,-1.2) node {$u$} (3.1,-1.8) node {$t$} (-1,1) node {$V_3$} (-1,-1.5) node {$V_2$} (-1,-4) node {$V_1$} (4,-1.5) node {$V_4$} (0,-5.3) node {$\A$} (3,-5.3) node {$\B$};
\draw[thick, draw=blue] (0,1) circle (0.7cm) (0,-1.5) circle (0.7cm) (0,-4) circle (0.7cm) (3,-1.5) circle (0.7cm);
\draw[thick] (-1.5,2) rectangle (4.5,-5) (2,2) -- (2,-5);
\end{scope}
\end{tikzpicture}
\caption{An example of $G$ (red) and $\HH$ (blue) with $r=4$, $s=2$, and $a=3$. Classes $V_1,V_2,V_3$ are accessible while $V_4$ is not. The vertex $w$ is movable to $V_1$, so it is a witness of the arc $V_2V_1$ with $V_2$ and $V_1$ being its home and target classes, respectively. The class $V_3$ is terminal since $V_1$ is reachable from $V_2$. The vertex $v$ is a solo vertex with nice solo neighbors $u$ and $t$. The vertex $y$ is ordinary since $z$ is movable to $V_2$. The class $V_4\in\B'(w)$. The value of $f_\emptyset(v)=f_\emptyset(v,u)+f_\emptyset(v,t)=1+1=2$.}
\label{definition-fig}
\end{figure}

\begin{definition}
\label{digraph-def}
Fix $r,s\in \mathbb{N}$. Let $G$ be a graph with $|V(G)|=rs-1$ and let $\vph$ be an $r$-coloring of $G$ with color classes $V_1,V_2,\dots,V_r$ such that $|V_1|=s-1$ and $|V_i|=s$ for every $i\in\{2,\ldots,r\}$. Let $\HH$ be the digraph whose vertex set is the color classes of $\vph$, and for distinct $i$ and $j$, the arc $V_iV_j$ is in $\HH$ if $N_{V_j}(v)=\emptyset$ for some $v\in V_i$. We call $v$ a \emph{witness} to the arc $V_iV_j$ and say $v$ is \emph{movable} to $V_j$. Further, we call $V_i$ and $V_j$ the \emph{home} and \emph{target} classes of $v$, respectively, in regards to this arc; see Figure~\ref{definition-fig}.

To avoid confusion, we call vertices of $\HH$ \emph{classes}. A class $V_i$ is \emph{reachable from} another class $V_j$ if there exists a $V_j,V_i$-path in $\HH$. A class $V_i$ is \emph{accessible} if $V_1$ is reachable from $V_i$. Observe that $V_1$ is trivially accessible. A class $V_i$ is \emph{terminal} if it is accessible, and every other accessible class $V_j$ remains accessible in $\HH\setminus V_i$; see Figure~\ref{definition-fig}. Clearly, $V_1$ is not terminal. Denote by $\A$, $\T$, and $\B$ the sets of accessible, terminal, and non-accessible classes, respectively. Let $a:=|\A|$ and $b:=|\B|$, and let $A:=\cup_{i=1}^aV_i$ and $B:=\cup_{i=a+1}^rV_i$. By symmetry, let $V_1,\dots,V_a$ be the classes in $\A$.

A vertex $v\in A$ with $v\in V_i$ is \emph{ordinary} if either (i) $a=1$ or (ii) $V_i\in\T$ and there exists $u\in V_i\setminus\{v\}$ which is movable to some other class in $\A$. Fix $v\in V_i$ and $u\in V_j$ with $V_i\in\A$ and $V_j\in \B$. We say $u$ is a \emph{solo neighbor} of $v$ if $v$ is the only neighbor of $u$ in $V_i$. Further, we say $v$ is a \emph{solo vertex} if it has at least one solo neighbor (in $B$). Let $Q(v)$ be the set of solo neighbors of $v$ and let $q(v):=|Q(v)|$. A solo neighbor $u_1$ of $v$ is \emph{nice} if there exists $u_2\in Q(v)$ such that $u_1u_2\notin E(G)$. By definition, $u_2$ is also nice. Let $Q'(v)$ be the set of nice solo neighbors of $v$ and let $q'(v):=|Q'(v)|$. Finally, for every $v\in A$, let $\B'(v)$ be the set of classes in $\B$ that have no neighbors of $v$; see Figure~\ref{definition-fig}.

Fix $V_i\in\A$ and $\W\subseteq \B$ (possibly $\W=\emptyset$) and let $W:=\cup_{V_j\in\W}V_j$. For every $v\in V_i$ and $w\in B$ with $vw\in E(H)$, let $f_{\W}(v,w):=\frac{1}{|N_{V_i}(w)|}$ if $w\in B\setminus W$ and $f_{\W}(v,w):=\frac{0.5}{|N_{V_i}(w)|}$ if $w\in W$. For every $v\in A$, let $f_{\W}(v):=\sum_{w\in N_B(v)}f_{\W}(v,w)$. If $W=\{V_i\}$ for some $i$, then we often write $f_{V_i}$ for $f_{\{V_i\}}$.
\end{definition}

\begin{obs}
\label{weight-sum}
By the definition of $\B$, every $w\in B$ has at least one neighbor in $V_i$ for every $V_i\in\A$. Thus, for every $V_i\in\A$ and $\W\subseteq\B$ with $W=\cup_{V_j\in\W}V_j$, we have $\sum_{v\in V_i}f_{\W}(v)=\sum_{w\in B, vw\in E(H)}f_{\W}(v,w)=|B\setminus W|+0.5|W|=s(|\B\setminus\W|+0.5|\W|)$.
\end{obs}

We now provide the promised proof sketch of the \hyperref[main-thm]{Main Theorem}. We start with a minimum counterexample $G$ to the \hyperref[main-thm]{Main Theorem}. We delete a vertex of low degree from $G$ to get a new graph $H$. The graph $H$ admits an equitable coloring $\vph$ by the minimality of $G$. We pick $\vph$ to maximize the value of $a$ from Definition~\ref{digraph-def}. Finally, we show that we can move some special vertices around to get a new coloring with a larger value of $a$, which gives a contradiction. 

To find those special vertices, which we call $v_1$, $v_2$, and $v_3$, we make use of the function $f$ from the end of Definition~\ref{digraph-def}. We typically consider $f_{\emptyset}$; although occasionally, we consider $f_{\W}$ for other sets $\W$. The vertex $v_1$ has small values of $f_\emptyset$ and $f_{\W}$ and is easy to find given the structure of $\A$. The vertices $v_2$ and $v_3$ have large values of $f_\emptyset$ and $f_{\W}$, and are guaranteed to exist because of the small $f$ value of $v_1$, using averaging arguments. 

Vertices with large values of $f_\emptyset$ and $f_{\W}$ are useful since nearly all of their neighbors must be solo neighbors in $B$, with exactly one such neighbor in most classes of $\B$. Nice solo neighbors, in particular, come in pairs with each pair sharing a class (Claim~\ref{two-solo-nbrs}). This allows for some leftover classes to which $v_2$ and $v_3$ are movable; i.e., $\B'(v_2)\neq\emptyset$ and $\B'(v_3)\neq\emptyset$. It also gives us substantial flexibility in that solo neighbors of $v_2$ and $v_3$ are then movable to $V_1$. This shuffling of vertices results in a ``better'' coloring (Claim~\ref{B'-nice-solo}), which contradicts our choice of $\vph$ as the best. 

It is also worth noting that we aim to shuffle vertices in such a way that most accessible classes remain accessible (possibly with respect to a new class $V_1$) even after shuffling. As a result, it suffices to identify one or two more classes that are accessible, in order to get a ``better" coloring. This is done by choosing $v_1$, $v_2$, and $v_3$ to be in a terminal class $V_i$ (Claim~\ref{dist-1-class}). Furthermore, $V_i$ itself remains accessible by the fact that $v_2$ and $v_3$ are ordinary vertices. In other words, there is another vertex (namely $v_1$) which acts as the initial ``witness" for a $V_i,V_1$-path.  

\section{Proof} 

\begin{proof}[Proof of the {\hyperref[main-thm]{Main Theorem}}]
Fix $r\geq13$ and consider the class $\mathcal{P}^r_1$ of 1-planar graphs with $\Delta(G)\le r$. Pick $G\in\mathcal{P}^r_1$ such that $G$ is a counterexample to the Main Theorem and $|E(G)|$ is minimum. By Lemma~\ref{divisibility-lem}, we may assume $|V(G)|=rs$ for some $s\ge1$. If $d(v)=0$ for every $v\in V(G)$, then the result follows trivially. So, fix $x\in V(G)$ with $1\le d(x)\le7$; note that $x$ must exist by Lemma~\ref{mindeg-lem}. Similarly, fix $y\in N(x)$. By the minimality of $|E(G)|$, there exists an equitable $r$-coloring $\psi$ of $G\setminus\{xy\}$. If $\psi(x)\neq\psi(y)$, then $\psi$ is also an equitable $r$-coloring of $G$, and we are done. So, assume $\psi(x)=\psi(y)$. Let $V_1,V_2,\dots,V_r$ be the color classes of $\psi$ with $x,y\in V_1$ and $|V_i|=s$ for every $i\in[r]$. Let $H:=G\setminus\{x\}$ and let $\vph$ be the restriction of $\psi$ to $H$. Let $\HH,\A,\B,A,B,a$ and $b$ be as in Definition~\ref{digraph-def}. Pick $\vph$ (more accurately, $\psi$) such that $a$ is maximum. 

\begin{claim}
\label{size-of-a}
The following holds: $a\le7$.
\end{claim}

\begin{proof}[Proof of Claim~\ref{size-of-a}]
To see this, assume otherwise, i.e., assume $a\ge8$. Since $d(x)\le7$, there exists $i\in[a]$ such that $x$ is movable to $V_i$. Moreover, by the definition of $\A$, there exists a $V_i,V_1$-path $P$ in $\HH$. Move the witness for each arc in $P$ from its home class to its target class. Now moving $x$ to $V_i$ gives an equitable $r$-coloring of $G$, a contradiction. 
\end{proof}

\begin{claim}
\label{A-B-edges}
The following holds: $abs\le |E(A,B)|\le 3(rs-1)-8$. In particular, $a\le4$.
\end{claim}

\begin{proof}[Proof of Claim~\ref{A-B-edges}]
To prove the first statement, note that the upper bound is given by Lemma~\ref{edges-lem}. Further, since every class in $\B$ is non-accessible, every vertex in $B$ has at least one neighbor in every class in $\A$. This proves the lower bound.

Now we prove the second statement. Recall that $a\le7$, by Observation~\ref{size-of-a}. By the first statement, $abs-3(rs-1)+6\le0$. However, if $5\le a\le7$, then since $b=r-a$ and $r\ge13$, $$abs-3(rs-1)+6=a(r-a)s-3rs+9=s[r(a-3)-a^2]+9>0,$$ a contradiction.
\end{proof}

\begin{claim}
\label{strong-comp}
There exists a strong component $\C$ in $\HH[\B]$ with $|\C|\ge r-4$.
\end{claim}

\begin{proof}[Proof of Claim~\ref{strong-comp}]
Assume every strong component of $\B$ has at most $r-5$ classes. Since $r\ge13$, by Claim~\ref{A-B-edges}, we have $b=r-a\ge r-4\ge 9$. This implies that there exist strong components of $\HH[\B]$ whose union $\U$ has at least 5 and at most $r-5$ classes. 
Let $u:=|\U|$, i.e., $u$ is the number of classes in $\U$. Further, let $U:=\cup_{V_i\in \U} V_i$, and $X:=\cup_{V_i \in \B\setminus\U}V_i$. 

Recall that $|V_i|=s$ for every $V_i\in\B$. Moreover, every vertex of every class of $\U$ has at least one neighbor in every class of $\A$, by the definition of $\B$. So, $|E(U,A)|\ge aus$. Further, for every $V_i\in \U$ and every $V_j\notin \U$, either $V_iV_j$ or $V_jV_i$ is not an arc of $\HH$; otherwise, $V_i$ and $V_j$ belong to the same strong component, a contradiction. Equivalently, either every vertex of $V_i$ has at least one neighbor in $V_j$ or vice versa. So, $|E(U,X)|\ge (b-u)us$. This implies $|E(U,V(H)\setminus U)|\ge (a+b-u)us=(r-u)us\ge 5(r-5)s=3rs+s(2r-25)>3(rs-1)-6$, contradicting Lemma~\ref{edges-lem}. 
\end{proof}

\begin{claim}
\label{q-q'}
Fix $v\in A$. If $q(v)\ge7$, then $q'(v)\geq q-3$.
\end{claim}

\begin{proof}[Proof of Claim~\ref{q-q'}]
Assume $q'(v)\le q(v)-4$. Consider the graph $H[Q(v)\cup\{v\}]$. Observe that $uw\in E(H)$ for every $u\in Q(v)\setminus Q'(v)$ and $w\in Q(v)\setminus\{u\}$, by the definition of $Q'(v)$. So, $|E(H[Q(v)\cup\{v\}])|=|E(K_{q(v)+1})|-|E(H[Q'(v)])|$. Recall that $|E(H[Q'(v)])|\le\binom{q'(v)}{2}$. So, if $q(v)\ge3$ and $q'(v)\le q(v)-4$, then 
\begin{align*}
|E(H[Q(v)\cup\{v\}])|
&=|E(K_{q(v)+1})|-|E(H[Q'(v)])| \\
&\ge\binom{q(v)+1}{2}-\binom{q'(v)}{2} \\
&\ge\binom{q(v)+1}{2}-\binom{q(v)-4}{2} \\
&=5q(v)-10>4(q(v)+1)-8,    
\end{align*} contradicting Lemma~\ref{edges-lem}.
\end{proof}

\begin{claim}
\label{two-solo-nbrs}
If $v$ is ordinary with $v\in V_i$ and $u$ is a nice solo neighbor of $v$ with $u\in V_j$, then $v$ has at least one more neighbor in $V_j$. 
\end{claim}

\begin{proof}[Proof of Claim~\ref{two-solo-nbrs}]
Since $v$ is ordinary, either (i) $a=1$ (so, $V_i=V_1$), or (ii) $V_i\in\T$ (so, $i\neq 1$) and there exists $z\in V_i\setminus\{v\}$ such that $z$ is movable to some other class $V_l$ ($l\neq i$) in $\A$. Assume $u$ is the only neighbor of $v$ in $V_j$. Since $u$ is a nice solo neighbor of $v$, there exists $w\in Q(v)$ such that $uw\notin E(H)$; see Figure~\ref{two-solo-nbrs-fig}. Let $V_k$ be the class of $w$. Note that $k\neq j$ since $u$ is the only neighbor of $v$ in $V_j$. Consider moving $v$ to $V_j$ and $u$ to $V_i$. 

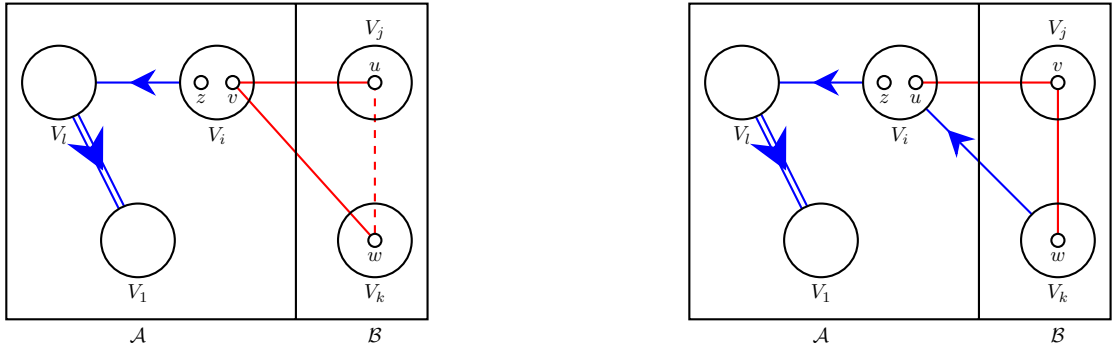
\begin{figure}[!h]
\centering
\begin{subfigure}{0.55\textwidth}
\begin{tikzpicture}[scale=0.7,every node/.style={scale=0.7}]
\draw[thick, dashed, red] (6,0) -- (6,-3);  
\draw[thick] (3.3,0) edge[red] (6,0) (3.3,0) edge[red] (6,-3) (2.3,0) edge[blue,decoration={markings, mark=at position 0.6 with {\arrow{Stealth[length=10pt,width=10pt]}}},postaction={decorate}] (0.7,0) (0.313,-0.626) edge[double, double distance= 0.5mm, blue, decoration={markings, mark=at position 0.6 with {\arrow{Stealth[length=15pt,width=15pt]}}},postaction={decorate}] (1.187,-2.374);
\draw[thick] (2.7,0) node[uStyle] {} (3.3,0) node[uStyle] {} (6,0) node[uStyle] {} (6,-3) node[uStyle] {} (2.7,-0.3) node {$z$} (3.3,-0.3) node {$v$} (6,0.3) node {$u$} (6,-3.3) node {$w$} (0,-1) node {$V_l$} (3,-1) node {$V_i$} (1.5,-4) node {$V_1$} (6,1) node {$V_j$} (6,-4) node {$V_k$} (1.5,-4.8) node {$\A$} (6,-4.8) node {$\B$};
\draw[thick] (0,0) circle (0.7cm) (3,0) circle (0.7cm) (1.5,-3) circle (0.7cm) (6,0) circle (0.7cm) (6,-3) circle (0.7cm);
\draw[thick] (-1,1.5) rectangle (7,-4.5) (4.5,1.5) -- (4.5,-4.5);
\end{tikzpicture}
\end{subfigure}%
\begin{subfigure}{0.35\textwidth}
\begin{tikzpicture}[scale=0.7,every node/.style={scale=0.7}]
\draw[thick] (3.3,0) edge[red] (6,0) (6,0) edge[red] (6,-3) (2.3,0) edge[blue,decoration={markings, mark=at position 0.6 with {\arrow{Stealth[length=10pt,width=10pt]}}},postaction={decorate}] (0.7,0) (5.505,-2.505) edge[blue,decoration={markings, mark=at position 0.8 with {\arrow{Stealth[length=10pt,width=10pt]}}},postaction={decorate}] (3.495,-0.495) (0.313,-0.626) edge[double, double distance= 0.5mm, blue, decoration={markings, mark=at position 0.6 with {\arrow{Stealth[length=15pt,width=15pt]}}},postaction={decorate}] (1.187,-2.374);
\draw[thick] (2.7,0) node[uStyle] {} (3.3,0) node[uStyle] {} (6,0) node[uStyle] {} (6,-3) node[uStyle] {} (2.7,-0.3) node {$z$} (3.3,-0.3) node {$u$} (6,0.3) node {$v$} (6,-3.3) node {$w$} (0,-1) node {$V_l$} (3,-1) node {$V_i$} (1.5,-4) node {$V_1$} (6,1) node {$V_j$} (6,-4) node {$V_k$} (1.5,-4.8) node {$\A$} (6,-4.8) node {$\B$};
\draw[thick] (0,0) circle (0.7cm) (3,0) circle (0.7cm) (1.5,-3) circle (0.7cm) (6,0) circle (0.7cm) (6,-3) circle (0.7cm);
\draw[thick] (-1,1.5) rectangle (7,-4.5) (4.5,1.5) -- (4.5,-4.5);
\end{tikzpicture}
\end{subfigure}
\caption{The red lines and red dashed lines are edges and non-edges of $H$, respectively. The blue lines and blue double lines are arcs and paths of $\HH$, respectively. Left: An example of $\HH$ in Claim~\ref{two-solo-nbrs} before moving $v$ and $u$. Right: The structure of $\HH$ in Claim~\ref{two-solo-nbrs} after moving $v$ and $u$.}
\label{two-solo-nbrs-fig}
\end{figure}

If we are in case (i), then note that $w$ is now movable to $(V_1\cup\{u\})\setminus\{v\}$ since $v$ was the only neighbor of $w$ in $V_1$ and $wu\notin E(H)$. This implies $V_k$ is now accessible which contradicts the maximality of $a$. Thus, assume instead we are in case (ii).  

Since $V_i\in\T$, every class $V_m\in\A$ with $m\neq i$ is still accessible. Further, $(V_i\cup\{u\})\setminus\{v\}$ is still accessible since $z$ is movable to $V_l$. However, as before, $w$ is now movable to $(V_i\cup\{u\})\setminus\{v\}$ since $v$ was the only neighbor of $w$ in $V_i$ and $wu\notin E(H)$. This implies $V_k$ is also accessible, which again contradicts the maximality of $a$.   
\end{proof}

\begin{claim}
\label{B'-nice-solo}
If $v$ is ordinary with $v\in V_i$ and $u$ is a nice solo neighbor of $v$ with $u\in V_j$, then $V_j$ is not reachable from $V_k$ for every $V_k\in \B'(v)$.
\end{claim}

\begin{figure}[h!]
\centering
\begin{subfigure}{0.55\textwidth}
\begin{tikzpicture}[scale=0.7,every node/.style={scale=0.7}]
\draw[thick, dashed, red] (6,-1.5) edge[bend left=45] (6,-4);
\draw[thick] (3.3,0) edge[red] (6,-1.5) (3.3,0) edge[red] (6,-4) (2.3,0) edge[blue,decoration={markings, mark=at position 0.6 with {\arrow{Stealth[length=10pt,width=10pt]}}},postaction={decorate}] (0.7,0) (0.313,-0.626) edge[double, double distance= 0.5mm, blue, decoration={markings, mark=at position 0.6 with {\arrow{Stealth[length=15pt,width=15pt]}}},postaction={decorate}] (1.187,-2.374) (6,0.3) edge[double, double distance= 0.5mm, blue, decoration={markings, mark=at position 0.7 with {\arrow{Stealth[length=12pt,width=12pt]}}},postaction={decorate}] (6,-0.8);
\draw[thick] (2.7,0) node[uStyle] {} (3.3,0) node[uStyle] {} (6,-1.5) node[uStyle] {} (6,-4) node[uStyle] {} (2.7,-0.3) node {$z$} (3.3,-0.3) node {$v$} (6,-1.2) node {$u$} (6,-4.3) node {$w$} (0,-1) node {$V_l$} (3,-1) node {$V_i$} (1.5,-4) node {$V_1$} (6,-2.5) node {$V_j$} (6,2) node {$V_k$} (6,-5) node {$V_m$} (1.5,-5.8) node {$\A$} (6,-5.8) node {$\B$};
\draw[thick] (0,0) circle (0.7cm) (3,0) circle (0.7cm) (1.5,-3) circle (0.7cm) (6,1) circle (0.7cm) (6,-1.5) circle (0.7cm) (6,-4) circle (0.7cm);
\draw[thick] (-1,2.5) rectangle (7,-5.5) (4.5,2.5) -- (4.5,-5.5);
\end{tikzpicture}
\end{subfigure}%
\begin{subfigure}{0.35\textwidth}
\begin{tikzpicture}[scale=0.7,every node/.style={scale=0.7}]
\draw[thick] (3.3,0) edge[red] (6,1) (6,1) edge[bend left=35, red] (6,-4) (2.3,0) edge[blue,decoration={markings, mark=at position 0.6 with {\arrow{Stealth[length=10pt,width=10pt]}}},postaction={decorate}] (0.7,0) (0.313,-0.626) edge[double, double distance= 0.5mm, blue, decoration={markings, mark=at position 0.6 with {\arrow{Stealth[length=15pt,width=15pt]}}},postaction={decorate}] (1.187,-2.374) (5.58,-3.44) edge[blue, decoration={markings, mark=at position 0.8 with {\arrow{Stealth[length=10pt,width=10pt]}}},postaction={decorate}] (3.42,-0.56);
\draw[thick] (2.7,0) node[uStyle] {} (3.3,0) node[uStyle] {} (6,1) node[uStyle] {} (6,-4) node[uStyle] {} (2.7,-0.3) node {$z$} (3.3,-0.3) node {$u$} (6,0.7) node {$v$} (6,-4.3) node {$w$} (0,-1) node {$V_l$} (3,-1) node {$V_i$} (1.5,-4) node {$V_1$} (6,-2.5) node {$V_j$} (6,2) node {$V_k$} (6,-5) node {$V_m$} (1.5,-5.8) node {$\A$} (6,-5.8) node {$\B$};
\draw[thick] (0,0) circle (0.7cm) (3,0) circle (0.7cm) (1.5,-3) circle (0.7cm) (6,1) circle (0.7cm) (6,-1.5) circle (0.7cm) (6,-4) circle (0.7cm);
\draw[thick] (-1,2.5) rectangle (7,-5.5) (4.5,2.5) -- (4.5,-5.5);
\end{tikzpicture}
\end{subfigure}
\caption{The red lines and red dashed lines are edges and non-edges of $H$, respectively. The blue lines and blue double lines are arcs and paths of $\HH$, respectively. Left: An example of $\HH$ in Claim~\ref{B'-nice-solo} before moving $v$, $u$, and each witness of $P$. Right: The structure of $\HH$ in Claim~\ref{B'-nice-solo} after moving $v$, $u$, and each witness of $P$.}
\label{B'-nice-solo-fig}
\end{figure}
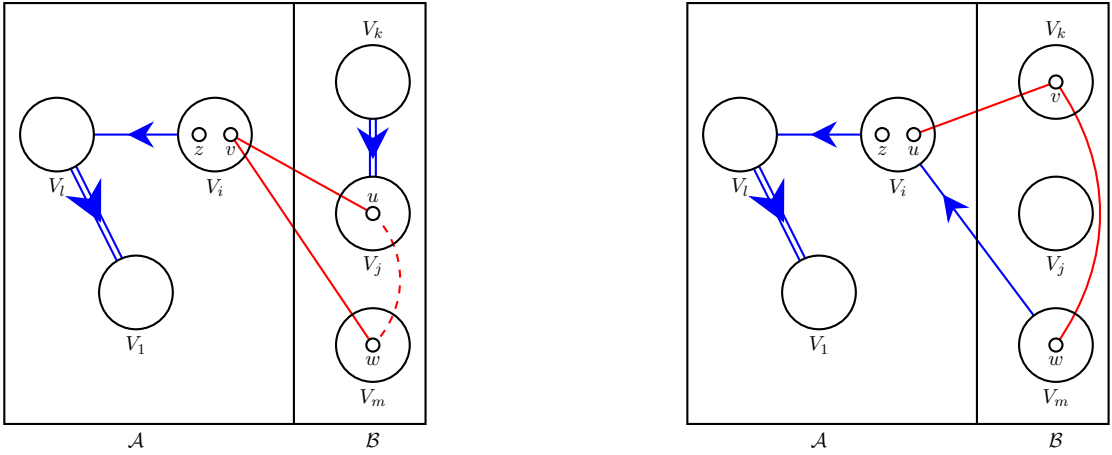

\begin{proof}[Proof of Claim~\ref{B'-nice-solo}]
Note that $V_k$ has no neighbors of $v$ since $V_k\in\B'(v)$. Moreover, $V_j$ has at least one neighbor of $v$ since $u\in V_j$. So, $k\neq j$. Since $v$ is ordinary, either (i) $a=1$ (so, $V_i=V_1$), or (ii) $V_i\in\T$ (so, $i\neq1$) and there exists $z\in V_i\setminus\{v\}$ such that $z$ is movable to some other class $V_l$ ($l\neq i$) in $\A$; see Figure~\ref{B'-nice-solo-fig}. Also, since $u$ is a nice solo neighbor of $v$, there exists $w\in Q(v)$ such that $uw\notin E(H)$. Let $V_m$ be the class of $w$ (possibly $m=j$). Assume there exists a $V_k,V_j$-path $P$ in $\HH$. We may assume $P$ does not contain $V_m$; otherwise, we exchange the roles of $u$ and $w$. Consider moving $v$ to $V_k$, moving $u$ to $V_i$, and for each arc in $P$ moving its witness from its home to its target class. 

If we are in case (i), then $w$ is now movable to $(V_1\cup\{u\})\setminus\{v\}$, since $v$ was the only neighbor of $w$ in $V_i$ and $wu\notin E(H)$. This implies $V_m$ is accessible which contradicts the maximality of $a$. Thus, assume instead we are in case (ii). 

Since $V_i\in\T$, every class $V_p\in\A$ with $p\neq i$ is still accessible. Further, $(V_i\cup\{u\})\setminus\{v\}$ is still accessible, since $z$ is movable to $V_l$. However, as before, $w$ is now movable to $(V_i\cup\{u\})\setminus\{v\}$ since $v$ was the only neighbor of $w$ in $V_i$ and $wu\notin E(H)$. This implies $V_m$ is also accessible, which contradicts the maximality of $a$. 
\end{proof}

The following two claims ensure that solo vertices do not have ``too many" neighbors in $B$. Hence, if $f$ is large for a solo vertex $v$, then $v$ must have many solo neighbors. In particular, we show in Claim~\ref{solo-vrt-has-nbr} that (almost) every solo vertex $v$ in a terminal class must have at least one neighbor in every other class of $\A$. However, we are unable to show that $v$ must have a neighbor in $V_1$ when $a=3$. Fortunately, it suffices to prove a weaker statement for that particular case, which we handle in Claim~\ref{solo-vrt-has-nbr-V1}.

\begin{claim}
\label{solo-vrt-has-nbr}
Fix a solo vertex $v\in V_i$ with $V_i\in\T$. If $a\in\{2,4\}$ (resp. $a=3$), then $v$ has at least one neighbor in $V_l$ for every $l\in[a]\setminus\{i\}$ (resp. $l\in[a]\setminus\{1,i\}$).  
\end{claim}

\begin{proof}

\begin{figure}[!h]
\centering
\begin{subfigure}{0.55\textwidth}
\begin{tikzpicture}[scale=0.7,every node/.style={scale=0.7}]
\draw[thick] (3,0) edge[red] (6,0) (2.3,0) edge[blue,decoration={markings, mark=at position 0.6 with {\arrow{Stealth[length=10pt,width=10pt]}}},postaction={decorate}] (0.7,0) (0.313,-0.626) edge[blue, decoration={markings, mark=at position 0.6 with {\arrow{Stealth[length=10pt,width=10pt]}}},postaction={decorate}] (1.187,-2.374);
\draw[thick] (0,0) node[uStyle] {} (3,0) node[uStyle] {} (6,0) node[uStyle] {} (0,-0.3) node {$w$} (3,-0.3) node {$v$} (6,-0.3) node {$u$} (0,-1) node {$V_l$} (3,-1) node {$V_i$} (1.5,-4) node {$V_1$} (6,-1) node {$V_j$} (1.5,-4.8) node {$\A$} (6,-4.8) node {$\B$};
\draw[thick] (0,0) circle (0.7cm) (3,0) circle (0.7cm) (1.5,-3) circle (0.7cm) (6,0) circle (0.7cm);
\draw[thick] (-1,1.5) rectangle (7,-4.5) (4.5,1.5) -- (4.5,-4.5);
\end{tikzpicture}
\end{subfigure}%
\begin{subfigure}{0.35\textwidth}
\begin{tikzpicture}[scale=0.7,every node/.style={scale=0.7}]
\draw[thick] (0,0) edge[red, bend left=45] (6,0) (5.3,0) edge[blue,decoration={markings, mark=at position 0.9 with {\arrow{Stealth[length=10pt,width=10pt]}}},postaction={decorate}] (3.7,0) (0.7,0) edge[blue,decoration={markings, mark=at position 0.6 with {\arrow{Stealth[length=10pt,width=10pt]}}},postaction={decorate}] (2.3,0) (1.187,-2.374) edge[blue, decoration={markings, mark=at position 0.6 with {\arrow{Stealth[length=10pt,width=10pt]}}},postaction={decorate}] (0.313,-0.626);
\draw[thick] (0,0) node[uStyle] {} (1.5,-3) node[uStyle] {} (6,0) node[uStyle] {} (1.5,-3.3) node {$w$} (0,-0.3) node {$v$} (6,-0.3) node {$u$} (0,-1) node {$V_l$} (3,-1) node {$V_i$} (1.5,-4) node {$V_1$} (6,-1) node {$V_j$} (1.5,-4.8) node {$\A$} (6,-4.8) node {$\B$};
\draw[thick] (0,0) circle (0.7cm) (3,0) circle (0.7cm) (1.5,-3) circle (0.7cm) (6,0) circle (0.7cm);
\draw[thick] (-1,1.5) rectangle (7,-4.5) (4.5,1.5) -- (4.5,-4.5);
\end{tikzpicture}
\end{subfigure}
\caption{The red lines are edges of $H$ and the blue lines are arcs of $\HH$. The case $a=3$ in Claim~\ref{solo-vrt-has-nbr} before moving $v$ and $w$ (left) and after moving $v$ and $w$ (right).}
\label{solo-vrt-has-nbr-fig}
\end{figure}
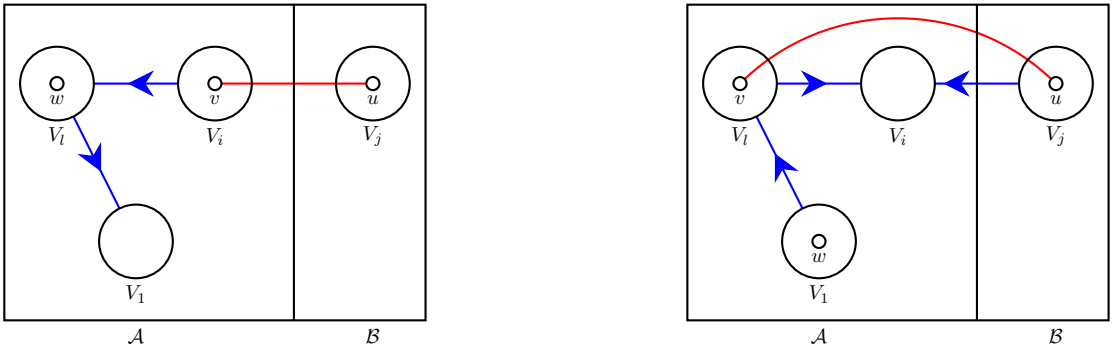

Note that $i\neq1$ since $V_i\in\T$. Let $u\in V_j$ be a solo neighbor of $v$. For $a\in\{2,4\}$ (resp. $a=3$), assume there exists $l\in[a]\setminus\{i\}$ (resp. $l\in[a]\setminus\{1,i\}$) such that $v$ has no neighbors in $V_l$. Since $V_i\in\T$, there exists a $V_l,V_1$-path $P$ in $\HH$ that does not contain $V_i$. Move $v$ to $V_l$ and move the witness of each arc in $P$ from its home to its target class. Call this new coloring $\vph'$. Now $|V_i\setminus\{v\}|=s-1$ and there is a $V_l(V_i\setminus\{v\})$ arc in $\HH$. Moreover, there is a $V_j(V_i\setminus\{v\})$ arc in $\HH$ since $u$ is a solo neighbor of $v$. In particular, $V_l$ and $V_j$ are accessible with respect to $V_i\setminus\{v\}$.

If $a=2$ in $\vph$, then $a\ge3$ in $\vph'$ which contradicts the maximality of $a$.

If $a=3$ in $\vph$, then observe that $P$ is an arc; see Figure~\ref{solo-vrt-has-nbr-fig}. Let $w$ be the witness of $P$. Since $v$ is movable to $V_l$, we have $vw\not\in E(H)$. So, there is a $V_1V_l$ arc in $\vph'$. This implies $V_1$ is also accessible with respect to $V_i\setminus\{v\}$ and $a\ge4$ in $\vph'$ which contradicts the maximality of $a$.

If $a=4$ in $\vph$, then $b=r-4\ge13-4=9$. Moreover, $\HH[\B]$ is strongly connected, by Claim~\ref{strong-comp}. So, there exists a $V_m,(V_i\setminus\{v\})$-path for every $a+1\le m\le r$ since $V_j(V_i\setminus\{v\})$ is an arc in $\vph'$. But now $a\ge11$ in $\vph'$ which contradicts the maximality of $a$.
\end{proof}

\begin{claim}
\label{solo-vrt-has-nbr-V1}
Fix $v\in V_i$ with $V_i\in\T$. If $a=3$ and $q(v)\ge7$, then $v$ has at least one neighbor in $V_1$.
\end{claim}

\begin{proof}[Proof of Claim~\ref{solo-vrt-has-nbr-V1}]
Assume $v$ is movable to $V_1$. Note that $i\neq1$ since $V_i\in\T$; by symmetry, say $i=2$. Since $a=3$ and $r\ge13$, we have $b=r-3\ge10$. Recall that $\HH[\B]$ contains a strong component $\U$ with $|\U|\ge r-4$. By symmetry, let $V_5,V_6,\dots,V_r\in V(\U)$; see Figure~\ref{solo-vrt-has-nbr-V1-fig}. 

Suppose there exists $u\in Q(v)$ with $u\in V_j$ for some $5\le j\le r$. Using a similar argument to Claim~\ref{solo-vrt-has-nbr}, move $v$ to $V_1$ and call the new coloring $\vph'$. Note that $|V_2\setminus\{v\}|=s-1$ in $\vph'$ and both $(V_1\cup\{v\})(V_2\setminus\{v\})$ and $V_j(V_2\setminus\{v\})$ are arcs in $\vph'$; that is, both are accessible with respect to $V_2\setminus\{v\}$. Moreover, since $\U$ is a strong component and $V_j\in V(\U)$, there is a $V_m,(V_2\setminus\{v\})$-path for every $m\in\{5,6,\dots,r\}\setminus\{j\}$. So, $a\ge11$ in $\vph'$ which contradicts the maximality of $a$. Thus, assume $Q(v)\subseteq V_4$.

\begin{figure}[!h]
\centering
\begin{subfigure}{0.55\textwidth}
\begin{tikzpicture}[scale=0.7,every node/.style={scale=0.7}]
\draw[thick, dashed, red] (6,-4) edge[bend left=45] (5.8,1.8) (6,-4) edge[bend left=40] (5.8,1.2);
\draw[thick] (0,0.3) edge[red, bend left=45] (3,0.3) (0,-0.3) edge[red, bend left=45] (3,-0.3) (3,0.3) edge[red] (5.8,1.8) (3,0.3) edge[red] (5.8,1.2) (3,-0.3) edge[red] (6,-4) (2.687,-0.626) edge[blue,decoration={markings, mark=at position 0.6 with {\arrow{Stealth[length=10pt,width=10pt]}}},postaction={decorate}] (1.813,-2.374) (0.313,-0.626) edge[blue, decoration={markings, mark=at position 0.6 with {\arrow{Stealth[length=10pt,width=10pt]}}},postaction={decorate}] (1.187,-2.374) (6,-2.2) edge[double, double distance= 0.5mm, blue, decoration={markings, mark=at position 0.7 with {\arrow{Stealth[length=12pt,width=12pt]}}},postaction={decorate}] (6,-3.3);
\draw[thick] (3,0.3) node[uStyle] {} (3,-0.3) node[uStyle] {} (5.8,1.8) node[uStyle] {} (5.8,1.2) node[uStyle] {} (6,-4) node[uStyle] {} (0,0.3) node[uStyle] {} (0,-0.3) node[uStyle] {} (2.7,0.3) node {$v$} (2.7,-0.3) node {$w$} (6.15,1.8) node {$z_1$} (6.15,1.2) node {$z_2$} (6,-4.3) node {$z_0$} (0,-1) node {$V_3$} (3,-1) node {$V_2$} (1.5,-4) node {$V_1$} (6,-0.6) node {$V_5\in\B'(v)$} (6,2.6) node {$V_4=\B'(w)$} (6,-5) node {$V_r$} (1.5,-5.8) node {$\A$} (6,-5.8) node {$\B$} (6,0.5) node {$Q(v)$};
\draw[thick] (0,0) circle (0.7cm) (3,0) circle (0.7cm) (1.5,-3) circle (0.7cm) (6,1.3) ellipse (1cm and 1.1cm) (6,1.5) ellipse (0.5cm and 0.7cm) (6,-1.5) circle (0.7cm) (6,-4) circle (0.7cm);
\draw[thick] (-1,3.2) rectangle (7.5,-5.5) (4.5,3.2) -- (4.5,-5.5);
\end{tikzpicture}
\end{subfigure}%
\begin{subfigure}{0.35\textwidth}
\begin{tikzpicture}[scale=0.7,every node/.style={scale=0.7}]
\draw[thick] (0,-0.3) edge[red, bend right=12] (6,-1.5) (0,0.3) edge[red] (5.8,1.8) (3,0.3) edge[red] (5.8,1.8) (6,-1.5) edge[red, bend left=70, looseness=1.5] (5.8,1.2) (3,-0.3) edge[red] (6,-1.5) (5.336,-1.721) edge[blue,decoration={markings, mark=at position 0.6 with {\arrow{Stealth[length=10pt,width=10pt]}}},postaction={decorate}] (2.164,-2.779) (0.313,-0.626) edge[blue, decoration={markings, mark=at position 0.6 with {\arrow{Stealth[length=10pt,width=10pt]}}},postaction={decorate}] (1.187,-2.374) (6,-3.3) edge[double, double distance= 0.5mm, blue, decoration={markings, mark=at position 0.7 with {\arrow{Stealth[length=12pt,width=12pt]}}},postaction={decorate}] (6,-2.2);
\draw[thick] (3,0.3) node[uStyle] {} (3,-0.3) node[uStyle] {} (5.8,1.8) node[uStyle] {} (5.8,1.2) node[uStyle] {} (6,-1.5) node[uStyle] {} (0,0.3) node[uStyle] {} (0,-0.3) node[uStyle] {} (2.65,0.3) node {$z_0$} (2.65,-0.3) node {$z_1$} (6.12,1.8) node {$w$} (6.15,1.2) node {$z_2$} (6,-1.8) node {$v$} (0,-1) node {$V_3$} (3,-1) node {$V_2$} (1.5,-4) node {$V_1$} (6,-0.6) node {$V_5\in\B'(v)$} (6,2.6) node {$V_4=\B'(w)$} (6,-5) node {$V_r$} (1.5,-5.8) node {$\A$} (6,-5.8) node {$\B$} (6,0.5) node {$Q(v)$};
\draw[thick] (0,0) circle (0.7cm) (3,0) circle (0.7cm) (1.5,-3) circle (0.7cm) (6,1.3) ellipse (1cm and 1.1cm) (6,1.5) ellipse (0.5cm and 0.7cm) (6,-1.5) circle (0.7cm) (6,-4) circle (0.7cm);
\draw[thick] (-1,3.2) rectangle (7.5,-5.5) (4.5,3.2) -- (4.5,-5.5);
\end{tikzpicture}
\end{subfigure}
\caption{The red lines and red dashed lines are edges and non-edges of $H$, respectively. The blue lines and blue double lines are arcs and paths of $\HH$, respectively. An example of $\vph$ (left) and $\vph'$ (right) in Claim~\ref{solo-vrt-has-nbr-V1}.}
\label{solo-vrt-has-nbr-V1-fig}
\end{figure}
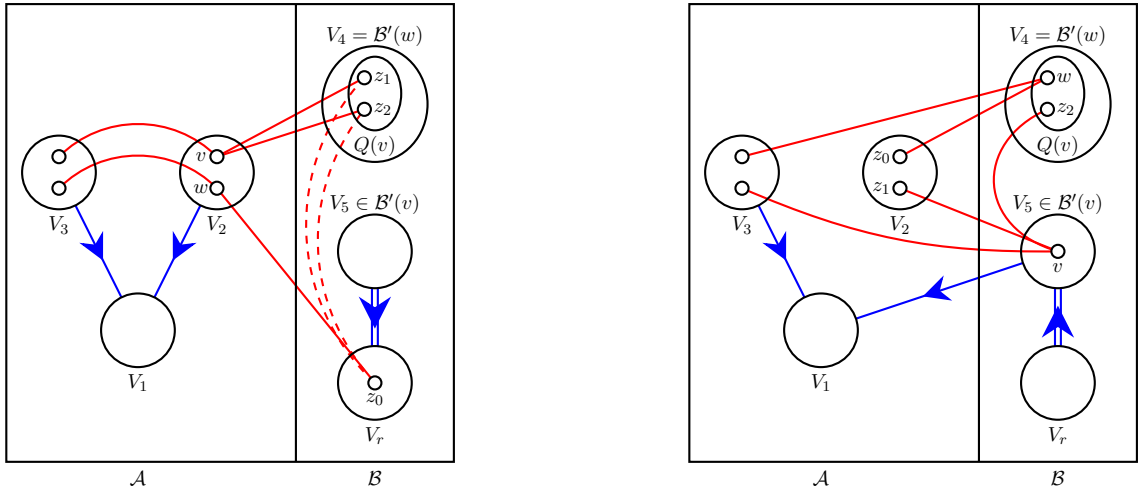

Note that $v$ has at least one neighbor in $V_3$, by Claim~\ref{solo-vrt-has-nbr}. By Observation~\ref{weight-sum}, we have $\frac{\sum_{w\in V_2}f_{V_4}(w)}{|V_2|}=\frac{s(r-4)+0.5s}{s}=r-3.5$. Further, since $Q(v)\subseteq V_4$ and $q(v)\ge7$, we have $f_{V_4}(v)\le (r-1)-\frac{7}{2}=r-4.5<r-3.5$. Thus, there exists $w\in V_2$ with $f_{V_4}(w)>r-3.5$.

Observe that the existence of $v$ implies $w$ is ordinary. Furthermore, $w$ is a solo vertex; otherwise, $f_{V_4}(w)\le \frac{r}{2}<r-3.5$. So, $w$ has at least one neighbor in $V_3$, by Claim~\ref{solo-vrt-has-nbr}. Let $m=|Q(w)\setminus V_4|$. If $m\le r-6$, then $f_{V_4}(w)\le m+0.5(r-m-1)\le r-6+0.5(6)-0.5=r-3.5$, a contradiction. So, assume $m\ge r-5$. Note that $m\ge8$ since $r\ge13$. This implies $q'(w)\ge5$, by Claim~\ref{q-q'}. Moreover, every $z\in Q'(w)$ is not the only neighbor of $w$ in its class, by Claim~\ref{two-solo-nbrs}. Thus, $|\B'(w)|\ge r-3-((r-1)-\lceil\frac{5}{2}\rceil)=1$.

Since $q'(w)\ge5$, if $Q'(w)\subseteq V_4$, then $f_{V_4}(w)\le r-1-\frac{5}{2}=r-3.5$, a contradiction. So, some $z\in Q'(w)$ is in $B\setminus V_4$. If there exists $V_j\in\B'(w)$ with $5\le j\le r$, then the class of $z$ is reachable from $V_j$ since $\U$ is strongly connected. However, this contradicts Claim~\ref{B'-nice-solo} (where $w$ and $z$ play the roles of $v$ and $u$ in Claim~\ref{B'-nice-solo}). Thus, assume $\B'(w)=V_4$. 

Observe that if $q'(w)\ge7$, then $|\B'(w)|\ge r-3-((r-1)-\lceil\frac{7}{2}\rceil)=2$ which contradicts $\B'(w)=V_4$. So, $5\le q'(w)\le 6$. Since $m\ge8$, there exists $z_0\in Q(w)\setminus Q'(w)$. By symmetry, let $V_r$ be the class of $z_0$ in $\U$; see Figure~\ref{solo-vrt-has-nbr-V1-fig}. Since $z_0\in Q(w)\setminus Q'(w)$, $\B'(w)=V_4$, and $m\ge r-5$, we have that $z_0$ is adjacent to $w$ and to at least $r-6$ vertices in $B\setminus V_4$. So, $z_0$ has at most 5 neighbors in $V_4$.  

Recall that $d(v)\le r$, $Q(v)\subseteq V_4$, and $q(v)\ge7$. This implies that there exists $z_1,z_2\in Q(v)$ such that $z_0$ is nonadjacent to both $z_1$ and $z_2$. Moreover, $\B'(v)\subseteq V(\U)$ and $|\B'(v)|\ge r-4-((r-1)-7)=4$. So, there exists $V_k\in \B'(v)$ with $k\neq r$; say $k=5$.

Since $\U$ is strongly connected, there is a $V_5,V_r$-path $P$ in $\HH$. Now move $w$ to $V_4$, move $v$ to $V_5$, move $z_1$ to $V_2$, move the witness of each arc of $P$ from its home to its target class, and move $z_0$ to $V_2$. Call the new coloring $\vph'$; see Figure~\ref{solo-vrt-has-nbr-V1-fig}. Observe that $V_3$ is still accessible since $V_2\in\T$. Also, each arc in $P$ is now reversed in $\vph'$, so there is a $V_r,V_5$-path $P'$ in $\vph'$. Observe that $V_5$ is accessible since, by assumption, $v$ is movable to $V_1$; thus, each class in $P'$ is also accessible. Now $a\ge4$ in $\vph'$ which contradicts the maximality of $a$. 
\end{proof}

\begin{claim}
\label{dist-1-class}
If $a\in\{3,4\}$, then we may assume there exists $V_i\in\T$ such that $V_iV_1$ is an arc in $\HH$. 
\end{claim}

\begin{proof}[Proof of Claim~\ref{dist-1-class}]
Recall that $\A$ contains a spanning in-tree $\I$ with sink $V_1$, by definition. We will show that either (i) $V_1$ has an in-neighbor $V_i$ that is a leaf in $\I$, or (ii) we can reverse an arc $V_jV_1$ in $\HH$, i.e.~move its witness $v$ from $V_j$ to $V_1$, to get a new coloring $\vph'$ and new in-tree $\I'$ with sink $V_j\setminus\{v\}$ where $V_j\setminus\{v\}$ has an in-neighbor that is a leaf in $\I'$. Note that if (i) is true, then $V_i$ satisfies the claim. Further, if (ii) is true and $\vph'$ has the same $a$ value as $\vph$, then $\vph'$ and $V_1\cup\{v\}$ satisfy the claim.


Pick a class whose shortest path $P$ to $V_1$ has maximum length, say $V_3$. Observe that $V_3\in\T$. If $V_3V_1$ is an arc in $\vph$, then $V_3$ satisfies the claim. So, assume otherwise. This implies $P$ has length at least 2. By symmetry, let $V_2$ be the class before last in $P$, i.e. $V_2V_1$ is an arc in $\vph$, and let $v$ be the witness of that arc. Move $v$ to $V_1$ and call this new coloring $\vph'$. Now $|V_2\setminus\{v\}|=s-1$, and $V_2\setminus\{v\}$ is a sink in an in-tree $\I'$ of $\vph'$. Moreover, every class in $P$ is accessible with respect to $V_2\setminus\{v\}$ in $\vph'$. 

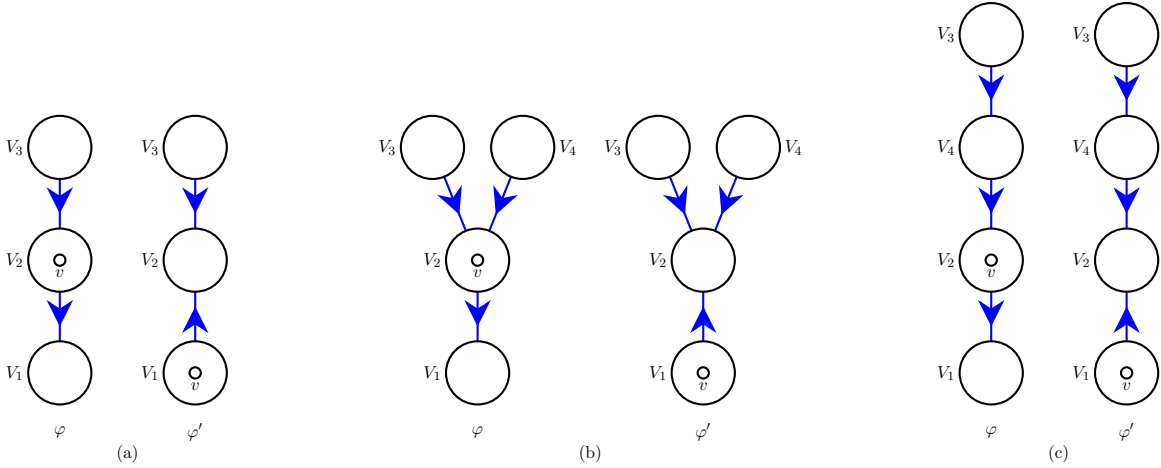
\begin{figure}[!h]
\centering
\begin{subfigure}{0.3\textwidth}
\begin{tikzpicture}[scale=0.6,every node/.style={scale=0.6}]
\begin{scope}
\draw[thick] (0,0.3) edge[blue, decoration={markings, mark=at position 0.7 with {\arrow{Stealth[length=10pt,width=10pt]}}},postaction={decorate}] (0,-0.8) (0,-2.2) edge[blue, decoration={markings, mark=at position 0.7 with {\arrow{Stealth[length=10pt,width=10pt]}}},postaction={decorate}] (0,-3.3);
\draw[thick] (0,-1.5) node[uStyle] {} (0,-1.8) node {$v$} (-1,1) node {$V_3$} (-1,-1.5) node {$V_2$} (-1,-4) node {$V_1$} (0,-5.3) node {$\vph$};
\draw[thick] (0,1) circle (0.7cm) (0,-1.5) circle (0.7cm) (0,-4) circle (0.7cm);
\end{scope}
\begin{scope}[xshift=3cm]
\draw[thick] (0,0.3) edge[blue, decoration={markings, mark=at position 0.7 with {\arrow{Stealth[length=10pt,width=10pt]}}},postaction={decorate}] (0,-0.8) (0,-3.3) edge[blue, decoration={markings, mark=at position 0.7 with {\arrow{Stealth[length=10pt,width=10pt]}}},postaction={decorate}] (0,-2.2);
\draw[thick] (0,-4) node[uStyle] {} (0,-4.3) node {$v$} (-1,1) node {$V_3$} (-1,-1.5) node {$V_2$} (-1,-4) node {$V_1$} (0,-5.3) node {$\vph'$};
\draw[thick] (0,1) circle (0.7cm) (0,-1.5) circle (0.7cm) (0,-4) circle (0.7cm);
\end{scope}
\draw[thick] (1.5,-5.8) node {(a)};
\end{tikzpicture}
\end{subfigure}%
\begin{subfigure}{0.45\textwidth}
\begin{tikzpicture}[scale=0.6,every node/.style={scale=0.6}]
\begin{scope}
\draw[thick] (-0.74,0.35) edge[blue, decoration={markings, mark=at position 0.7 with {\arrow{Stealth[length=10pt,width=10pt]}}},postaction={decorate}] (-0.26,-0.85) (0.74,0.35) edge[blue, decoration={markings, mark=at position 0.7 with {\arrow{Stealth[length=10pt,width=10pt]}}},postaction={decorate}] (0.26,-0.85) (0,-2.2) edge[blue, decoration={markings, mark=at position 0.7 with {\arrow{Stealth[length=10pt,width=10pt]}}},postaction={decorate}] (0,-3.3);
\draw[thick] (0,-1.5) node[uStyle] {} (0,-1.8) node {$v$} (-2,1) node {$V_3$} (2,1) node {$V_4$} (-1,-1.5) node {$V_2$} (-1,-4) node {$V_1$} (0,-5.3) node {$\vph$};
\draw[thick] (-1,1) circle (0.7cm) (1,1) circle (0.7cm) (0,-1.5) circle (0.7cm) (0,-4) circle (0.7cm);
\end{scope}
\begin{scope}[xshift=5cm]
\draw[thick] (-0.74,0.35) edge[blue, decoration={markings, mark=at position 0.7 with {\arrow{Stealth[length=10pt,width=10pt]}}},postaction={decorate}] (-0.26,-0.85) (0.74,0.35) edge[blue, decoration={markings, mark=at position 0.7 with {\arrow{Stealth[length=10pt,width=10pt]}}},postaction={decorate}] (0.26,-0.85) (0,-3.3) edge[blue, decoration={markings, mark=at position 0.7 with {\arrow{Stealth[length=10pt,width=10pt]}}},postaction={decorate}] (0,-2.2);
\draw[thick] (0,-4) node[uStyle] {} (0,-4.3) node {$v$} (-2,1) node {$V_3$} (2,1) node {$V_4$} (-1,-1.5) node {$V_2$} (-1,-4) node {$V_1$} (0,-5.3) node {$\vph'$};
\draw[thick] (-1,1) circle (0.7cm) (1,1) circle (0.7cm) (0,-1.5) circle (0.7cm) (0,-4) circle (0.7cm);
\end{scope}
\draw[thick] (2.5,-5.8) node {(b)};
\end{tikzpicture}
\end{subfigure}%
\begin{subfigure}{0.2\textwidth}
\begin{tikzpicture}[scale=0.6,every node/.style={scale=0.6}]
\begin{scope}
\draw[thick] (0,2.8) edge[blue, decoration={markings, mark=at position 0.7 with {\arrow{Stealth[length=10pt,width=10pt]}}},postaction={decorate}] (0,1.7) (0,0.3) edge[blue, decoration={markings, mark=at position 0.7 with {\arrow{Stealth[length=10pt,width=10pt]}}},postaction={decorate}] (0,-0.8) (0,-2.2) edge[blue, decoration={markings, mark=at position 0.7 with {\arrow{Stealth[length=10pt,width=10pt]}}},postaction={decorate}] (0,-3.3);
\draw[thick] (0,-1.5) node[uStyle] {} (0,-1.8) node {$v$} (-1,1) node {$V_4$} (-1,3.5) node {$V_3$} (-1,-1.5) node {$V_2$} (-1,-4) node {$V_1$} (0,-5.3) node {$\vph$};
\draw[thick] (0,3.5) circle (0.7cm) (0,1) circle (0.7cm) (0,-1.5) circle (0.7cm) (0,-4) circle (0.7cm);
\end{scope}
\begin{scope}[xshift=3cm]
\draw[thick] (0,2.8) edge[blue, decoration={markings, mark=at position 0.7 with {\arrow{Stealth[length=10pt,width=10pt]}}},postaction={decorate}] (0,1.7) (0,0.3) edge[blue, decoration={markings, mark=at position 0.7 with {\arrow{Stealth[length=10pt,width=10pt]}}},postaction={decorate}] (0,-0.8) (0,-3.3) edge[blue, decoration={markings, mark=at position 0.7 with {\arrow{Stealth[length=10pt,width=10pt]}}},postaction={decorate}] (0,-2.2);
\draw[thick] (0,-4) node[uStyle] {} (0,-4.3) node {$v$} (-1,1) node {$V_4$} (-1,3.5) node {$V_3$} (-1,-1.5) node {$V_2$} (-1,-4) node {$V_1$} (0,-5.3) node {$\vph'$};
\draw[thick] (0,3.5) circle (0.7cm) (0,1) circle (0.7cm) (0,-1.5) circle (0.7cm) (0,-4) circle (0.7cm);
\end{scope}
\draw[thick] (1.5,-5.8) node {(c)};
\end{tikzpicture}
\end{subfigure}%
\caption{The different cases of Claim~\ref{dist-1-class}. (a-c): The left figure shows the structure of $\I$ in $\vph$ and the right figure shows the structure of $\I'$ in $\vph'$.}
\label{dist-1-class-fig}
\end{figure}

If $a=3$ in $\vph$, then also $a=3$ in $\vph'$ and $V_1\cup\{v\}$ satisfies the claim; see Figure~\ref{dist-1-class-fig}(a). 

If $a=4$ in $\vph$, then let $P'$ be a $V_4,V_1$-path in $\vph$. If $P'\subset P$ or $P'$ shares a class with $P$ (other than $V_1$), then $V_4$ is accessible with respect to $V_2\setminus\{v\}$; see Figure~\ref{dist-1-class-fig}(b-c). This implies $a=4$ and $V_1\cup\{v\}$ satisfies the claim in $\vph'$. Thus, assume $P'$ is disjoint from $P$ (except for $V_1$). Now $P'$ is an arc $V_4V_1$. But, this implies $V_4$ satisfies the claim in $\vph$.  
\end{proof}

Recall by Claim~\ref{A-B-edges} that $a\le4$. We split the rest of the proof into 4 cases depending on the value of $a$. The proof for each case starts off in a similar manner following the scheme explained at the end of Section~\ref{prelims}. However, note that each subsequent case introduces a new obstacle. This is mainly because, as $a$ gets smaller, 1) the number of classes that (possibly) do not belong to some strong component $\U$ in $\HH[\B]$ increases, and 2) lower bounds for certain quantities decrease. Nevertheless, we provide workarounds for each obstacle.

\textbf{Case 1: $\bm{a=4}$.} This implies $b=r-4\ge9$ since $r\ge13$. By Claim~\ref{dist-1-class}, we may assume there exists a class in $\T$, say $V_2$, such that $V_2V_1$ is an arc in $\HH$. So, $V_2$ contains a vertex $v_1$ movable to $V_1$. This implies $q(v_1)=0$, by Claim~\ref{solo-vrt-has-nbr}. Recall that $d(v)\le r$ for every $v\in V(H)$. Now $f_\emptyset(v_1)\le\frac{r}{2}<r-4$ since $r\ge13$. Moreover, $\frac{\sum_{v\in V_2}f_\emptyset(v)}{|V_2|}=\frac{s(r-4)}{s}=r-4$, by Observation~\ref{weight-sum}. Thus, there exists $v_2\in V_2$ with $f_\emptyset(v_2)>r-4$. 

Observe that the existence of $v_1\in V_2$ implies that $v_2$ is ordinary. Furthermore, $v_2$ is a solo vertex; otherwise, $f_\emptyset(v_2)\le\frac{r}{2}<r-4$, a contradiction. By Claim~\ref{solo-vrt-has-nbr}, there is at least one neighbor of $v_2$ in each of $V_1$, $V_3$, and $V_4$. So, $|N_B(v_2)|\le d(v_2)-3\le r-3$. Note that $f_\emptyset(v_2)\le q(v_2)+0.5|N_B(v_2)\setminus Q(v_2)|=0.5(|N_B(v_2)|+q(v_2))$. Since $f_\emptyset(v_2)>r-4$ and $r\ge13$, we have $|N_B(v_2)|=r-3$ and $q(v_2)\ge r-4\ge9$. This implies $q'(v_2)\ge6$, by Claim~\ref{q-q'}. Furthermore, every $w\in Q'(v_2)$ is not the only neighbor of $v_2$ in its class, by Claim~\ref{two-solo-nbrs}. Thus, $|\B'(v_2)|\ge r-4-((r-3)-\frac{6}{2})=2$. 

Let $w\in Q'(v_2)$ and $V_i\in\B'(v_2)$. Recall that $b=r-4$; by Claim~\ref{strong-comp}, this implies $\HH[\B]$ is strongly connected. Now the class of $w$ is reachable from $V_i$, contradicting Claim~\ref{B'-nice-solo}. 

\textbf{Case 2: $\bm{a=3}$.} This implies $b=r-3\ge10$ since $r\ge13$. By Claim~\ref{dist-1-class}, we may assume there exists a class in $\T$, say $V_2$, such that $V_2V_1$ is an arc in $\HH$. So, $V_2$ contains a vertex $v_1$ movable to $V_1$. Note that $q(v_1)\le6$, by Claim~\ref{solo-vrt-has-nbr-V1}. Now $f_\emptyset(v_1)\le\frac{r}{2}+\frac{6}{2}<r-3$ since $r\ge13$. Moreover, $\frac{\sum_{v\in V_2}f_\emptyset(v)}{|V_2|}=\frac{s(r-3)}{s}=r-3$, by Observation~\ref{weight-sum}. Thus, there exists $v_2\in V_2$ with $f_\emptyset(v_2)>r-3$. 

As in Case 1, again $v_2$ is ordinary because of $v_1$. Furthermore, $q(v_2)\ge7$; otherwise, $f_\emptyset(v_2)\le\frac{r}{2}+\frac{6}{2}<r-3$, a contradiction. By Claims~\ref{solo-vrt-has-nbr} and \ref{solo-vrt-has-nbr-V1}, there is at least one neighbor of $v_2$ in each of $V_1$ and $V_3$. So, $|N_B(v_2)|\le r-2$. Similar to the previous case, $f_\emptyset(v_2)\le 0.5(|N_B(v_2)|+q(v_2))$. Since $f_\emptyset(v_2)>r-3$ and $r\ge13$, we have $|N_B(v_2)|=r-2$ and $q(v_2)\ge r-3\ge10$. This implies $q'(v_2)\ge7$, by Claim~\ref{q-q'}. Furthermore, every $w\in Q'(v_2)$ is not the only neighbor of $v_2$ in its class, by Claim~\ref{two-solo-nbrs}. Thus, $|\B'(v_2)|\ge r-3-((r-2)-\lceil\frac{7}{2}\rceil)=3$. 

Recall that there exists a strong component of $\HH[\B]$, say $\U$, with at least $r-4$ classes. By symmetry, let $V_5,V_6,\dots,V_r\in V(\U)$. Since $|\B'(v_2)|\ge3$, there exists $V_i\in \B'(v_2)$ such that $V_i\in V(\U)$, by Pigeonhole. If some vertex $w$ of $Q'(v_2)$ is in a class of $\U$, then its class is reachable from $V_i$, contradicting Claim~\ref{B'-nice-solo}. Thus, assume $Q'(v_2)\subseteq V_4$. 

By Observation~\ref{weight-sum}, note that $\frac{\sum_{v\in V_2}f_{V_4}(v)}{|V_2|}=\frac{s(r-4)+s(0.5)}{s}=r-3.5$. Recall from above that $q(v_1)\le6$. If $q(v_1)\le5$, then $f_{V_4}(v_1)\le\frac{r}{2}+\frac{5}{2}<r-3.5$ since $r\ge13$. Similarly, if $q(v_1)=6$, then $v_1$ has at least one neighbor in $V_3$, by Claim~\ref{solo-vrt-has-nbr}; so, $f_{V_4}(v_1)\le\frac{r-1}{2}+\frac{6}{2}<r-3.5$. Thus, there exists $v_3\in V_2$ with $f_{V_4}(v_3)>r-3.5$.

Since $Q'(v_2)\subseteq V_4$, $q'(v_2)\ge7$, and $q(v_2)\le r$, we have $f_{V_4}(v_2)\le |Q(v_2)\setminus Q'(v_2)|+0.5q'(v_2)+0.5(r-q(v_2))=0.5(q(v_2)-q'(v_2)+r)\le r-0.5q'(v_2)\le r-3.5$. Furthermore, recall from above that $f_{V_4}(v_1)\le f_\emptyset(v_1)\le \frac{r}{2}+\frac{6}{2}\le r-3.5$. Thus, $v_3\neq v_2$ and $v_3\neq v_1$.

Observe that $v_3$ is ordinary because of $v_1$ and has at least one neighbor in $V_3$, by Claim~\ref{solo-vrt-has-nbr}. Also, $v_3$ has at least one neighbor in $V_1$; otherwise, $q(v_3)\le6$, by Claim~\ref{solo-vrt-has-nbr-V1}, and $f_{V_4}(v_3)\le\frac{r-1}{2}+\frac{6}{2}<r-3.5$, a contradiction. Thus, $|N_B(v_3)|\le r-2$. Let $m:=|Q(v_3)\setminus V_4|$. Now $f_{V_4}(v_3)\le m+0.5(r-m-2)$. If $m\le r-5$, then $f_{V_4}(v_3)\le r-5+0.5(5)-0.5(2)=r-3.5$, a contradiction. So, assume $m\ge r-4$; thus, $m\ge9$ since $r\ge13$. Now $q'(v_3)\ge6$, by Claim~\ref{q-q'}. Moreover, every $w\in Q'(v_3)$ is not the only neighbor of $v_3$ in its class, by Claim~\ref{two-solo-nbrs}. Thus, $|\B'(v_3)|\ge r-3-((r-2)-\frac{6}{2})=2$.

By Pigeonhole, some $V_j\in\B'(v_3)$ is in $V(\U)$. Now if six vertices of $Q'(v_3)$ are in $V_4$, then $f_{V_4}(v_3)\le (r-2)-\frac{6}{2}=r-5<r-3.5$, a contradiction. So, at least one vertex $w\in Q'(v_3)$ is in some class of $\U$. Since $\U$ is strongly connected, the class of $w$ is reachable from $V_j$, contradicting Claim~\ref{B'-nice-solo}. 

\textbf{Case 3: $\bm{a=2}$.} This implies $b=r-2\ge11$ since $r\ge13$. Observe that $V_2\in\T$ and $V_2V_1$ is an arc in $\HH$. So, $V_2$ contains a vertex $v_1$ movable to $V_1$. This implies $q(v_1)=0$, by Claim~\ref{solo-vrt-has-nbr}. Now $f_\emptyset(v_1)\le\frac{r}{2}<r-2$ since $r\ge13$. Moreover, $\frac{\sum_{v\in V_2}f_\emptyset(v)}{|V_2|}=\frac{s(r-2)}{s}=r-2$, by Observation~\ref{weight-sum}. Thus, there exists $v_2\in V_2$ with $f_\emptyset(v_2)>r-2$. 

As above, $v_2$ is ordinary because of $v_1$. Furthermore, $v_2$ is a solo vertex; otherwise, $f_\emptyset(v_2)\le\frac{r}{2}<r-2$, a contradiction. By Claim~\ref{solo-vrt-has-nbr}, there is at least one neighbor of $v_2$ in $V_1$. So, $|N_B(v_2)|\le r-1$. Similar to previous cases, $f_\emptyset(v_2)\le 0.5(|N_B(v_2)|+q(v_2))$. Since $f_\emptyset(v_2)>r-2$ and $r\ge13$, we have $|N_B(v_2)|=r-1$ and $q(v_2)\ge r-2\ge11$. This implies $q'(v_2)\ge8$, by Claim~\ref{q-q'}. Furthermore, every $w\in Q'(v_2)$ is not the only neighbor of $v_2$ in its class, by Claim~\ref{two-solo-nbrs}. Thus, $|\B'(v_2)|\ge r-2-((r-1)-\frac{8}{2})=3$. 

Recall from Claim~\ref{strong-comp} that there exists a strong component of $\HH[\B]$, say $\U$, with at least $r-4$ classes. By symmetry, let $V_5,V_6,\dots,V_r\in V(\U)$. Since $|\B'(v_2)|\ge3$, there exists $V_i\in \B'(v_2)$ such that $V_i\in V(\U)$, by Pigeonhole. If some vertex $w$ of $Q'(v_2)$ is in a class of $\U$, then its class is reachable from $V_i$, contradicting Claim~\ref{B'-nice-solo}. Thus, assume $Q'(v_2)\subseteq (V_3\cup V_4)$. 

Let $\W:=\{V_3,V_4\}$. By Observation~\ref{weight-sum}, note that $\frac{\sum_{v\in V_2}f_{\W}(v)}{|V_2|}=\frac{s(r-4)+2s(0.5)}{s}=r-3$. Moreover, $f_{\W}(v_1)\le\frac{r}{2}<r-3$ since $r\ge13$. Thus, there exists $v_3\in V_2$ with $f_{\W}(v_3)>r-3$. 

Since $Q'(v_2)\subseteq (V_3\cup V_4)$, $q'(v_2)\ge8$, and $q(v_2)\le r$, we have $f_{\W}(v_2)\le |Q(v_2)\setminus Q'(v_2)|+0.5q'(v_2)+0.5(r-q(v_2))=0.5(q(v_2)-q'(v_2)+r)\le r-0.5q'(v_2)\le r-4<r-3$. Furthermore, recall from above that $f_{V_4}(v_1)\le f_\emptyset(v_1)\le \frac{r}{2}< r-3$. Thus, $v_3\neq v_2$ and $v_3\neq v_1$.

Observe that $v_3$ is ordinary because of $v_1$. Further, $v_3$ is a solo vertex; otherwise $f_{\W}(v_3)\le\frac{r}{2}<r-3$, a contradiction. Let $m:=|Q(v_3)\setminus(V_3\cup V_4)|$. Recall that $v_3$ has at least one neighbor in $V_1$, by Claim~\ref{solo-vrt-has-nbr}. So, $f_{\W}(v_3)\le m+0.5(r-m-1)$. If $m\le r-5$, then $f_{\W}(v_3)\le r-5+0.5(5)-0.5=r-3$, a contradiction. So, assume $m\ge r-4$; thus, $m\ge9$ since $r\ge13$. Now $q'(v_3)\ge6$, by Claim~\ref{q-q'}. Also, recall that every $w\in Q'(v_3)$ is not the only neighbor of $v_3$ in its class, by Claim~\ref{two-solo-nbrs}. Thus, $|\B'(v_3)|\ge r-2-((r-1)-\frac{6}{2})=2$.

If at least 6 vertices of $Q'(v_3)$ are in $(V_3\cup V_4)$, then $f_{\W}(v_3)\le r-\frac{6}{2}=r-3$, a contradiction. So, at least one vertex $w\in Q'(v_3)$ is in some class of $\U$. If some $V_j\in\B'(v_3)$ is in $V(\U)$, then the class of $w$ is reachable from $V_j$, contradicting Claim~\ref{B'-nice-solo}. Thus, $\B'(v_3)\subseteq \{V_3,V_4\}$; in particular, $\B'(v_3)=\{V_3,V_4\}$ since $|\B'(v_3)|\ge2$. 

Observe that if $q'(v_3)\ge7$, then $|\B'(v_3)|\ge r-2-((r-1)-\lceil\frac{7}{2}\rceil)=3$ which contradicts $\B'(v_3)=\{V_3,V_4\}$. So, assume $q'(v_3)=6$. Since $m\ge9$, there exists $z_0\in Q(v_3)\setminus Q'(v_3)$. By symmetry, let $V_r$ be the class of $z_0$ in $\U$; see Figure~\ref{case-a=2-fig}. Since $z_0\in Q(v_3)\setminus Q'(v_3)$, $\B'(v_3)=\{V_3,V_4\}$, and $m\ge r-4$, we have that $z_0$ is adjacent to $v_3$ and at least $r-5$ vertices in $B\setminus(V_3\cup V_4)$. So, $z_0$ has at most 4 neighbors in $V_3\cup V_4$. 

Recall that $Q'(v_2)\subseteq (V_3\cup V_4)$ and $q'(v_2)\ge8$. This implies that there exists at least 4 vertices in $Q'(v_2)$ each of which is nonadjacent to $z_0$. By Pigeonhole, at least two such vertices share a class, say $z_1$ and $z_2$ in $V_3$; see Figure~\ref{case-a=2-fig}. 

Recall that $V_i\in\B'(v_2)$ and $5\le i\le r$. Since $\U$ is strongly connected, there is a $V_i,V_r$-path $P$ in $\HH[\U]$. Now move $v_3$ to $V_3$, move $v_2$ to $V_i$, move $z_1$ to $V_2$, move $z_0$ to $V_2$, and move the witness of each arc in $P$ from its home to its target class. Call the new coloring $\vph'$; see Figure~\ref{case-a=2-fig}. Observe that $V_2$ is still accessible since $v_1$ is movable to $V_1$. Furthermore, $V_3$ is accessible since $z_2$ is movable to $V_2$. Now $a\ge3$ in $\vph'$ which contradicts the maximality of $a$. 

\begin{figure}[!h]
\centering
\begin{subfigure}{0.55\textwidth}
\begin{tikzpicture}[scale=0.7,every node/.style={scale=0.7}]
\draw[thick, dashed, red] (7.5,-4) edge[out=135, in=-135] (5.8,1.6) (7.5,-4) edge[bend left=10] (5.8,1);
\draw[thick] (2.9,0.4) edge[red] (5.8,1.6) (2.9,0.4) edge[red] (5.8,1) (2.9,-0.4) edge[red] (7.5,-4) (3,-1) edge[blue,decoration={markings, mark=at position 0.6 with {\arrow{Stealth[length=10pt,width=10pt]}}},postaction={decorate}] (3,-2.3) (7.5,-2.2) edge[double, double distance= 0.5mm, blue, decoration={markings, mark=at position 0.7 with {\arrow{Stealth[length=12pt,width=12pt]}}},postaction={decorate}] (7.5,-3.3);
\draw[thick] (3.3,0) node[uStyle] {} (2.9,0.4) node[uStyle] {} (2.9,-0.4) node[uStyle] {} (5.8,1.6) node[uStyle] {} (5.8,1) node[uStyle] {} (7.5,-4) node[uStyle] {} (3.7,0) node {$v_1$} (2.5,0.4) node {$v_2$} (2.5,-0.4) node {$v_3$} (6.2,1.6) node {$z_1$} (6.2,1) node {$z_2$} (7.5,-4.3) node {$z_0$} (1.7,0) node {$V_2$} (2,-3) node {$V_1$} (9.2,-1.5) node {$V_i\in\B'(v_2)$} (7.5,-0.2) node {$\B'(v_3)$} (7,1.3) node {$V_3$} (8,1.3) node {$V_4$} (8.5,-4) node {$V_r$} (3,-5.3) node {$\A$} (7.5,-5.3) node {$\B$};
\draw[thick] (3,0) circle (1cm) (3,-3) circle (0.7cm) (6,1.3) circle (0.7cm) (9,1.3) circle (0.7cm) (7.5,-1.5) circle (0.7cm) (7.5,-4) circle (0.7cm);
\draw[thick] (1.2,2.5) rectangle (10.3,-5) (4.5,2.5) -- (4.5,-5);
\draw[thick, decorate, decoration={calligraphic brace, mirror, amplitude=3mm}] (5.3,0.5) -- (9.7,0.5);
\end{tikzpicture}
\end{subfigure}%
\begin{subfigure}{0.35\textwidth}
\begin{tikzpicture}[scale=0.7,every node/.style={scale=0.7}]
\draw[thick] (7.5,-1.5) edge[red] (2.9,-0.4) (7.5,-1.5) edge[red] (5.8,1) (5.8,1.6) edge[red] (2.9,0.4) (3,-1) edge[blue,decoration={markings, mark=at position 0.6 with {\arrow{Stealth[length=10pt,width=10pt]}}},postaction={decorate}] (3,-2.3) (5.358,1.022) edge[blue, decoration={markings, mark=at position 0.5 with {\arrow{Stealth[length=12pt,width=12pt]}}},postaction={decorate}] (3.918,0.398);
\draw[thick] (3.3,0) node[uStyle] {} (2.9,0.4) node[uStyle] {} (2.9,-0.4) node[uStyle] {} (5.8,1.6) node[uStyle] {} (5.8,1) node[uStyle] {} (7.5,-1.5) node[uStyle] {} (3.7,0) node {$v_1$} (2.5,0.4) node {$z_0$} (2.5,-0.4) node {$z_1$} (6.2,1.6) node {$v_3$} (6.2,1) node {$z_2$} (7.5,-1.9) node {$v_2$} (1.7,0) node {$V_2$} (2,-3) node {$V_1$} (9.2,-1.5) node {$V_i\in\B'(v_2)$} (7.5,-0.2) node {$\B'(v_3)$} (7,1.3) node {$V_3$} (8,1.3) node {$V_4$} (8.5,-4) node {$V_r$} (3,-5.3) node {$\A$} (7.5,-5.3) node {$\B$};
\draw[thick] (3,0) circle (1cm) (3,-3) circle (0.7cm) (6,1.3) circle (0.7cm) (9,1.3) circle (0.7cm) (7.5,-1.5) circle (0.7cm) (7.5,-4) circle (0.7cm);
\draw[thick] (1.2,2.5) rectangle (10.3,-5) (4.5,2.5) -- (4.5,-5);
\draw[thick, decorate, decoration={calligraphic brace, mirror, amplitude=3mm}] (5.3,0.5) -- (9.7,0.5);
\end{tikzpicture}
\end{subfigure}
\caption{The red lines and red dashed lines are edges and non-edges of $H$, respectively. The blue lines and blue double lines are arcs and paths of $\HH$, respectively. An example of $\vph$ (left) and $\vph'$ (right) in Case 3.}
\label{case-a=2-fig}
\end{figure}
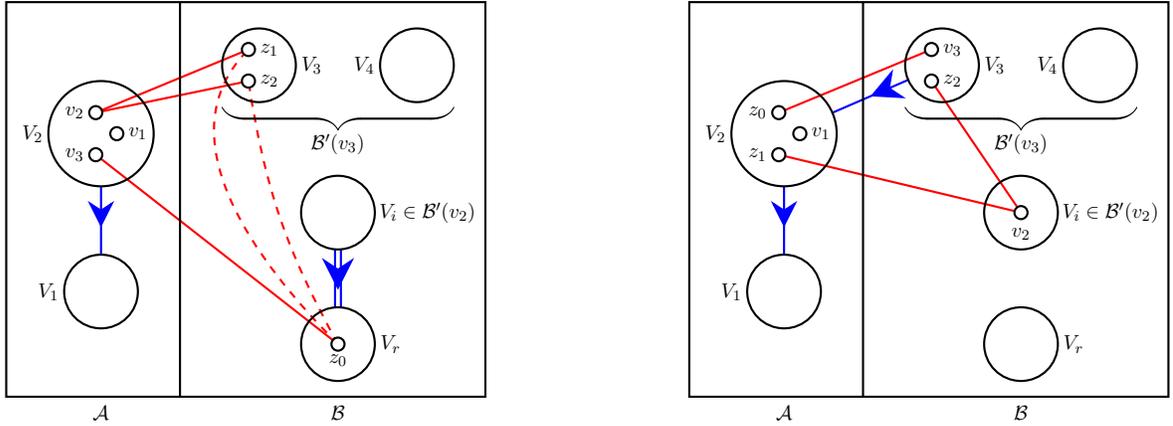

\textbf{Case 4: $\bm{a=1}$.} This implies $b=r-1\ge12$ since $r\ge13$. Observe that $\frac{\sum_{v\in V_1}f_\emptyset(v)}{|V_1|}=\frac{s(r-1)}{s-1}>r-1$. Thus, there exists $v_2\in V_1$ with $f_\emptyset(v_2)>r-1$. 

Note that $N(v_2)\subseteq B$ and $v_2$ is ordinary since $a=1$. Furthermore, $v_2$ is a solo vertex; otherwise, $f_\emptyset(v_2)\le\frac{r}{2}<r-1$, a contradiction. Similar to previous cases, $f_\emptyset(v_2)\le 0.5(|N(v_2)|+q(v_2))$. Since $f_\emptyset(v_2)>r-1$ and $r\ge13$, we have $|N_B(v_2)|=r$ and $q(v_2)\ge r-1\ge12$. This implies $q'(v_2)\ge9$, by Claim~\ref{q-q'}. Furthermore, every $w\in Q'(v_2)$ is not the only neighbor of $v_2$ in its class, by Claim~\ref{two-solo-nbrs}. Thus, $|\B'(v_2)|\ge r-1-(r-\lceil\frac{9}{2}\rceil)=4$. 

Recall from Claim~\ref{strong-comp} that there exists a strong component of $\HH[\B]$, say $\U$, with at least $r-4$ classes. By symmetry, let $V_5,V_6,\dots,V_r\in V(\U)$. Since $|\B'(v_2)|\ge4$, there exists $V_i\in \B'(v_2)$ such that $V_i\in V(\U)$. If some vertex $w$ of $Q'(v_2)$ is in a class of $\U$, then its class is reachable from $V_i$, contradicting Claim~\ref{B'-nice-solo}. Thus, assume $Q'(v_2)\subseteq (V_2\cup V_3\cup V_4)$. 

Let $\W:=\{V_2,V_3,V_4\}$. Note that $\frac{\sum_{v\in V_1}f_{\W}(v)}{|V_1|}=\frac{s(r-4)+3s(0.5)}{s-1}=\frac{s(r-2.5)}{s-1}>r-2.5$. So, there exists $v_3\in V_1$ with $f_{\W}(v_3)>r-2.5$. Furthermore, since $Q'(v_2)\subseteq (V_2\cup V_3\cup V_4)$, $q'(v_2)\ge9$, and $q(v_2)\le r$, we have $f_{\W}(v_2)\le |Q(v_2)\setminus Q'(v_2)|+0.5q'(v_2)+0.5(r-q(v_2))=0.5(q(v_2)-q'(v_2)+r)\le r-4.5<r-2.5$. Thus, $v_3\neq v_2$.

Again, $v_3$ is ordinary since $a=1$. Further, $v_3$ is a solo vertex; otherwise $f_{\W}(v_3)\le\frac{r}{2}<r-2.5$, a contradiction. Let $m:=|Q(v_3)\setminus(V_2\cup V_3\cup V_4)|$. Note that $f_{\W}(v_3)\le m+0.5(r-m)$. If $m\le r-5$, then $f_{\W}(v_3)\le r-5+0.5(5)=r-2.5$, a contradiction. So, assume $m\ge r-4$; thus, $m\ge9$ since $r\ge13$. Now $q'(v_3)\ge6$, by Claim~\ref{q-q'}. Also, recall that every $w\in Q'(v_3)$ is not the only neighbor of $v_3$ in its class, by Claim~\ref{two-solo-nbrs}. Thus, $|\B'(v_3)|\ge r-1-(r-\frac{6}{2})=2$.

If at least 6 vertices of $Q'(v_3)$ are in $V_2\cup V_3\cup V_4$, then $f_{\W}(v_3)\le r-\frac{6}{2}=r-3<r-2.5$, a contradiction. So, at least one vertex $w\in Q'(v_3)$ is in some class of $\U$. If some $V_j\in\B'(v_3)$ is in $V(\U)$, then the class of $w$ is reachable from $V_j$, contradicting Claim~\ref{B'-nice-solo}. Thus, $\B'(v_3)\subseteq\{V_2,V_3,V_4\}$. Since $|\B'(v_3)|\ge2$, by symmetry we assume, $V_2,V_3\in\B'(v_3)$. 

\begin{claim}
\label{subclaim}
$\B'(v_3)=\{V_2,V_3,V_4\}$, i.e., also $V_4\in\B'(v_3)$. 
\end{claim}

\begin{proof}[Proof of Claim~\ref{subclaim}]
Assume $V_4\notin\B'(v_3)$. Recall from above that $q(v_3)\ge m\ge9$. If $q(v_3)\ge10$, then $q'(v_3)\ge7$, by Claim~\ref{q-q'}. Moreover, as before, $|\B'(v_3)|\ge r-1-(r-\lceil\frac{7}{2}\rceil)=3$, by Claim~\ref{two-solo-nbrs}. So, $\B'(v_3)=\{V_2,V_3,V_4\}$ since $\B'(v_3)\subseteq\{V_2,V_3,V_4\}$, a contradiction. Thus, assume $q(v_3)=9$. 

Now $q'(v_3)=6$, $m=9$, and $Q(v_3)\subseteq B\setminus(V_2\cup V_3\cup V_4)$, by the definition of $m$. Moreover, $r=13$; otherwise, $m\ge r-4\ge 10$, a contradiction. Also, $|N_{V_4}(v_3)|\le1$; otherwise, $f_{\W}(v_3)\le 2(0.25)+2(0.5)+9=10.5\le r-2.5$, a contradiction. So, $N_{V_4}(v_3)=1$ since $V_4\notin\B'(v_3)$. Further, recall that every $v\in Q'(v_3)$ is not the only neighbor of $v_3$ in its class, by Claim~\ref{two-solo-nbrs}. Since no class of $\U$ is in $\B'(v_3)$ and $q'(v_3)=6$, the rest of the neighbors of $v_3$ are distributed as shown in Figure~\ref{subclaim-fig} (left). In particular, $V_5,V_6,V_7$ each contain a pair of nice solo neighbors of $v_3$.

Let $\D:=\{V_5,V_6,V_7\}$ and let $w_1$ and $w_2$ be the two nice solo neighbors of $v_3$ in $V_5$. Among classes in $\D$, let $V_5$ be a class such that the length of a shortest $V_8,V_5$-path $P$ in $\HH[\U]$ is minimized. Let $V_kV_5$ be the last arc in $P$ with witness $z$.

\begin{figure}[t!]
\centering
\begin{subfigure}{0.5\textwidth}
\begin{tikzpicture}[scale=0.7,every node/.style={scale=0.7}]
\draw[thick] (0.2,-1.25) node[uStyle] {} (-0.2,-1.25) node {$v_3$} (2.2,-0.3) node[uStyle, gray] {} (2.8,-0.3) node[uStyle, gray] {} (2.2,-0.7) node {$w_1$} (2.8,-0.7) node {$w_2$} (3.8,-0.5) node[uStyle, gray] {} (4.2,-0.5) node[uStyle, gray] {} (5.8,-0.5) node[uStyle, gray] {} (6.2,-0.5) node[uStyle, gray] {} (2.5,-2) node[uStyle, black] {} (4,-2) node[uStyle, black] {} (6,-2) node[uStyle, black] {} (2.5,-3.5) node[uStyle] {} (4,-3.5) node[uStyle] {} (6,-3.5) node[uStyle] {} (6,1.5) node[uStyle] {};
\draw[thick] (0,-2.05) node {$V_1$} (1.7,1.5) node {$V_2$} (1.5,-0.5) node {$V_5$} (1.7,-2) node {$V_8$} (1.7,-3.5) node {$V_{11}$} (4.8,1.5) node {$V_3$} (4.8,-0.5) node {$V_6$} (4.8,-2) node {$V_9$} (4.8,-3.5) node {$V_{12}$} (6.8,1.5) node {$V_4$} (6.8,-0.5) node {$V_7$} (6.8,-2) node {$V_{10}$} (6.8,-3.5) node {$V_{13}$};
\draw[thick] (0,-1.25) circle (0.5cm) (2.5,1.5) circle (0.5cm) (4,1.5) circle (0.5) (6,1.5) circle (0.5cm) (2.5,-0.5) circle (0.7cm) (4,-0.5) circle (0.5) (6,-0.5) circle (0.5cm) (2.5,-2) circle (0.5cm) (4,-2) circle (0.5) (6,-2) circle (0.5cm) (2.5,-3.5) circle (0.5cm) (4,-3.5) circle (0.5) (6,-3.5) circle (0.5cm);
\draw[thick] (-1,2.5) rectangle (7.5,-4.5) (1,2.5) -- (1,-4.5) (0,-4.8) node {$\A$} (4.25,-4.8) node {$\B$} (7.8,-2) node {$\U$}; 
\draw[gray] (1,0.5) -- (7.5,0.5);
\end{tikzpicture}
\end{subfigure}%
\begin{subfigure}{0.4\textwidth}
\begin{tikzpicture}[scale=0.7,every node/.style={scale=0.7}]
\draw[thick] (2,0) edge[double, double distance= 0.5mm, blue, decoration={markings, mark=at position 0.7 with {\arrow{Stealth[length=12pt,width=12pt]}}},postaction={decorate}] (3,0) (4,0) edge[double, double distance= 0.5mm, blue, decoration={markings, mark=at position 0.7 with {\arrow{Stealth[length=12pt,width=12pt]}}},postaction={decorate}] (5,0) (6,0) edge[blue, decoration={markings, mark=at position 0.7 with {\arrow{Stealth[length=10pt,width=10pt]}}},postaction={decorate}] (6.8,0);
\draw[thick] (-0.3,0) node[uStyle] {} (-0.7,0) node {$v_3$} (1.5,0) node[uStyle] {} (3.6,0) node[uStyle] {} (3.3,0) node {$v$} (5.6,0) node[uStyle] {} (5.3,0) node {$z$} (7.2,0.2) node[uStyle] {} (7.8,0.2) node[uStyle] {} (7.2,-0.2) node {$w_1$} (7.8,-0.2) node {$w_2$} (-0.5,-0.8) node {$V_1$} (1.5,-0.8) node {$V_8$} (3.5,-0.8) node {$V_j$} (5.5,-0.8) node {$V_k$} (7.5,-1) node {$V_5$} (3.5,1.3) node {$\vph$};
\draw[thick] (-0.5,0) circle (0.5cm) (1.5,0) circle (0.5cm) (3.5,0) circle (0.5cm) (5.5,0) circle (0.5cm) (7.5,0) circle (0.7cm);
\draw[thick] (-1.5,1) rectangle (8.5,-1.5) (0.5,1) -- (0.5,-1.5);
\end{tikzpicture}
\vspace{0.5mm}

\begin{tikzpicture}[scale=0.7,every node/.style={scale=0.7}]
\draw[white] (0,2) node {};
\draw[thick] (0.2,0) edge[bend left=35, blue, decoration={markings, mark=at position 0.55 with {\arrow{Stealth[length=10pt,width=10pt]}}},postaction={decorate}] (7,0);
\draw[thick] (-0.2,0.2) node[uStyle] {} (-0.8,0.2) node[uStyle] {} (-0.2,-0.2) node {$w_1$} (-0.8,-0.2) node {$w_2$} (1.5,0) node[uStyle] {} (3.7,0) node[uStyle] {} (3.3,0) node {$v_3$} (7.6,0) node[uStyle] {} (7.3,0) node {$z$} (-0.5,-1) node {$V_1$} (1.5,-0.8) node {$V_8$} (3.5,-0.8) node {$V_j$} (5.5,-0.8) node {$V_k$} (7.5,-0.8) node {$V_5$} (3.5,-1.8) node {$\vph'$};
\draw[thick] (-0.5,0) circle (0.7cm) (1.5,0) circle (0.5cm) (3.5,0) circle (0.5cm) (5.5,0) circle (0.5cm) (7.5,0) circle (0.5cm);
\draw[thick] (-1.5,1.5) rectangle (8.5,-1.5) (0.5,1.5) -- (0.5,-1.5);
\end{tikzpicture}

\end{subfigure}
\caption{Left: All vertices shown in $\B$ are neighbors of $v_3$ (edges are omitted to avoid clutter). Gray vertices are in $Q'(v_3)$, black vertices are in $Q(v_3)\setminus Q'(v_3)$, and white vertices are not in $Q(v_3)$. Right: An example of $\vph$ (top) and $\vph'$ (bottom) in Claim~\ref{subclaim}.}
\label{subclaim-fig}
\end{figure}
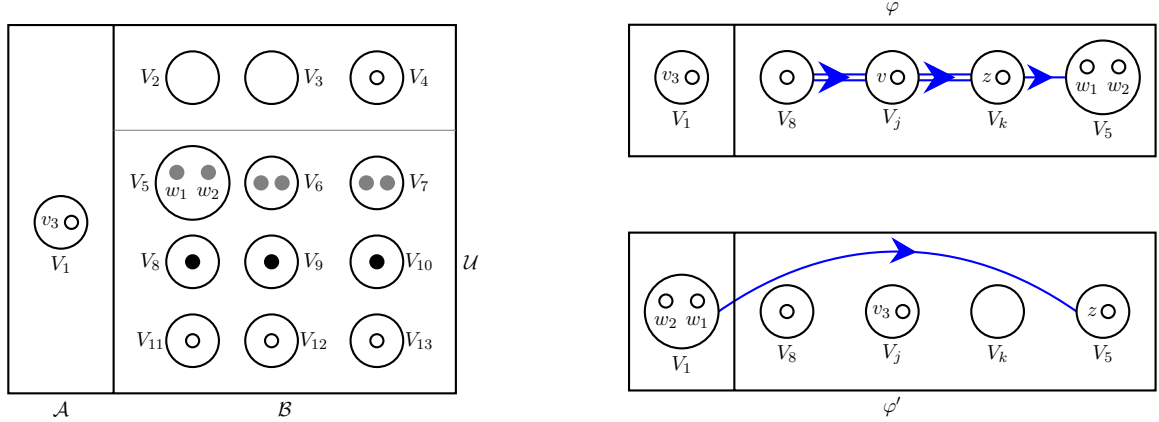

Suppose each witness of an arc in $P$ is not in $N(v_3)$. Now move $w_1$ and $w_2$ to $V_1$ and move $v_3$ to $V_5$. Call the new coloring $\vph'$. Let $V'_5:=(V_5\setminus\{w_1,w_2\})\cup\{v_3\}$. Note that $|V'_5|=s-1$ and $V_k$ is accessible with respect to $V'_5$ since $z$ is movable to $V'_5$; this contradicts the maximality of $a$. 

Suppose instead that there exists $V_j$ in $V(P)$ with witness $v\in N(v_3)$ of arc $V_jV_{j+1}$; see Figure~\ref{subclaim-fig} (top right). Observe that $v\notin Q'(v_3)$; otherwise, $V_j\in\D$ and has a shorter path from $V_8$, which contradicts the choice of $V_5$. Move the witness of each arc in the $V_jV_5$-subpath from its home to its target class, then move $v_3$ to $V_j$ and move $w_1$ and $w_2$ to $V_1$. Call the new coloring $\vph'$; see Figure~\ref{subclaim-fig} (bottom right). Let $V'_5:=(V_5\setminus\{w_1,w_2\})\cup\{z\}$. Now $|V'_5|=s-1$ and $(V_1\setminus\{v_3\})\cup\{w_1,w_2\}$ is accessible with respect to $V'_5$ simply by moving $w_1$ or $w_2$ to $V'_5$; this contradicts the maximality of $a$.
\end{proof} 

\begin{figure}[h!]
\centering
\begin{subfigure}{0.55\textwidth}
\begin{tikzpicture}[scale=0.7,every node/.style={scale=0.7}]
\draw[thick, dashed, red] (7.5,-4) edge[bend left=45] (5.4,1.6) (7.5,-4) edge[bend left=25] (5.4,1);
\draw[thick] (2.9,-1.1) edge[red] (5.4,1.6) (2.9,-1.1) edge[red] (5.4,1) (2.9,-1.9) edge[red] (7.5,-4) (7.5,-2.2) edge[double, double distance= 0.5mm, blue, decoration={markings, mark=at position 0.7 with {\arrow{Stealth[length=12pt,width=12pt]}}},postaction={decorate}] (7.5,-3.3);
\draw[thick] (3.3,-1.5) node[uStyle] {} (2.9,-1.1) node[uStyle] {} (2.9,-1.9) node[uStyle] {} (5.4,1.6) node[uStyle] {} (5.4,1) node[uStyle] {} (7.5,-4) node[uStyle] {} (3.7,-1.5) node {$v_1$} (2.5,-1.1) node {$v_2$} (2.5,-1.9) node {$v_3$} (5.8,1.6) node {$z_1$} (5.8,1) node {$z_2$} (7.5,-4.3) node {$z_0$} (1.7,-1.5) node {$V_1$} (9.2,-1.5) node {$V_i\in\B'(v_2)$} (7.5,-0.2) node {$\B'(v_3)$} (5.6,2.3) node {$V_2$} (7.5,2.3) node {$V_3$} (9.4,2.3) node {$V_4$} (8.5,-4) node {$V_r$} (3,-5.3) node {$\A$} (7.5,-5.3) node {$\B$};
\draw[thick] (3,-1.5) circle (1cm) (5.6,1.3) circle (0.7cm) (9.4,1.3) circle (0.7cm) (7.5,1.3) circle (0.7cm) (7.5,-1.5) circle (0.7cm) (7.5,-4) circle (0.7cm);
\draw[thick] (1.3,2.7) rectangle (10.4,-5) (4.5,2.7) -- (4.5,-5);
\draw[thick, decorate, decoration={calligraphic brace, mirror, amplitude=3mm}] (5.3,0.5) -- (9.7,0.5);
\end{tikzpicture}
\end{subfigure}%
\begin{subfigure}{0.35\textwidth}
\begin{tikzpicture}[scale=0.7,every node/.style={scale=0.7}]
\draw[thick] (7.5,-1.5) edge[red, bend left=25] (2.9,-1.9) (7.5,-1.5) edge[red] (5.4,1) (2.9,-1.1) edge[red] (5.4,1.6) (5.124,0.787) edge[blue, decoration={markings, mark=at position 0.8 with {\arrow{Stealth[length=10pt,width=10pt]}}},postaction={decorate}] (3.68,-0.767);
\draw[thick] (3.3,-1.5) node[uStyle] {} (2.9,-1.1) node[uStyle] {} (2.9,-1.9) node[uStyle] {} (5.4,1.6) node[uStyle] {} (5.4,1) node[uStyle] {} (7.5,-1.5) node[uStyle] {} (3.7,-1.5) node {$v_1$} (2.5,-1.1) node {$z_0$} (2.5,-1.9) node {$z_1$} (5.8,1.6) node {$v_3$} (5.8,1) node {$z_2$} (7.5,-1.8) node {$v_2$} (1.7,-1.5) node {$V_1$} (9.2,-1.5) node {$V_i\in\B'(v_2)$} (7.5,-0.2) node {$\B'(v_3)$} (5.6,2.3) node {$V_2$} (7.5,2.3) node {$V_3$} (9.4,2.3) node {$V_4$} (8.5,-4) node {$V_r$} (3,-5.3) node {$\A$} (7.5,-5.3) node {$\B$};
\draw[thick] (3,-1.5) circle (1cm) (5.6,1.3) circle (0.7cm) (9.4,1.3) circle (0.7cm) (7.5,1.3) circle (0.7cm) (7.5,-1.5) circle (0.7cm) (7.5,-4) circle (0.7cm);
\draw[thick] (1.3,2.7) rectangle (10.4,-5) (4.5,2.7) -- (4.5,-5);
\draw[thick, decorate, decoration={calligraphic brace, mirror, amplitude=3mm}] (5.3,0.5) -- (9.7,0.5);
\end{tikzpicture}
\end{subfigure}
\caption{The red lines are edges of $H$ and the blue lines and blue double lines are arcs and paths of $\HH$, respectively. An example of $\vph$ (left) and $\vph'$ (right) in Case 4.}
\label{case-a=1-fig}
\end{figure}
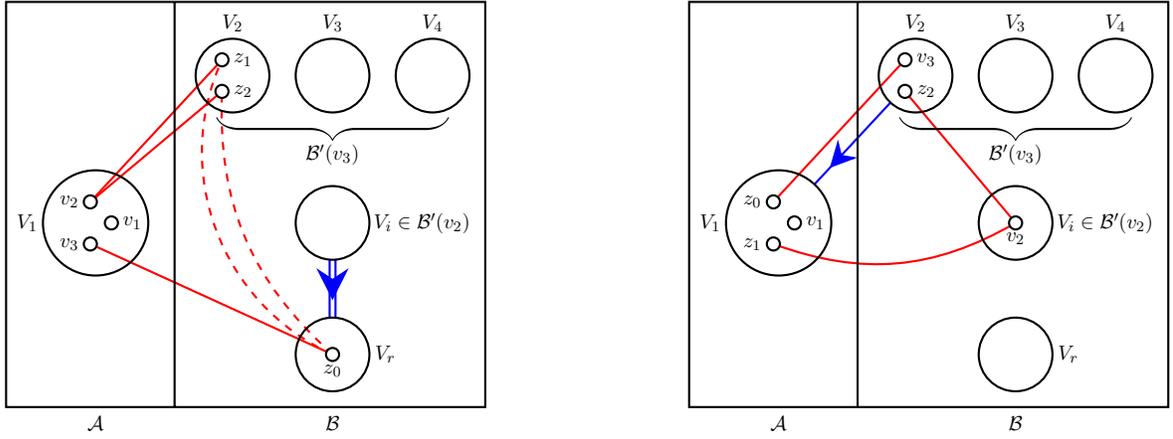

Observe that if $q'(v_3)\ge9$, then $|\B'(v_3)|\ge r-1-(r-\lceil\frac{9}{2}\rceil)=4$ which contradicts $\B'(v_3)=\{V_2,V_3,V_4\}$. So, assume $6\le q'(v_3)\le 8$. Since $m\ge9$, there exists $z_0\in Q(v_3)\setminus Q'(v_3)$. By symmetry, let $V_r$ be the class of $z_0$ in $\U$; see Figure~\ref{case-a=1-fig}. Since $z_0\in Q(v_3)\setminus Q'(v_3)$, $\B'(v_3)=\{V_2,V_3,V_4\}$, and $m\ge r-4$, we have that $z_0$ is adjacent to $v_3$ and at least $r-5$ vertices in $B\setminus(V_2\cup V_3\cup V_4)$. So, $z_0$ has at most 4 neighbors in $V_2\cup V_3\cup V_4$. 

Recall that $Q'(v_2)\subseteq(V_2\cup V_3\cup V_4)$ and $q'(v_2)\ge9$. This implies that there exist at least 5 vertices in $Q'(v_2)$ each of which is nonadjacent to $z_0$. By Pigeonhole, at least two such vertices share a class, say $z_1$ and $z_2$ in $V_2$; see Figure~\ref{case-a=1-fig}. 

Recall, from the third paragraph of this case, that $V_i\in\B'(v_2)$ and $5\le i\le r$. Since $\U$ is strongly connected, there is a $V_i,V_r$-path $P$ in $\HH[\U]$. Now move $v_3$ to $V_2$, move $v_2$ to $V_i$, move $z_0$ and $z_1$ to $V_1$, and move the witness of each arc in $P$ from its home to its target class. Call the new coloring $\vph'$; see Figure~\ref{case-a=1-fig}. Observe that $V_2$ is accessible since $z_2$ is movable to $V_1$. Now $a\ge2$ in $\vph'$ which contradicts the maximality of $a$. 
\end{proof}

\bibliographystyle{plainurl}
\footnotesize{
\bibliography{references}

\begin{thebibliography}{10}

\bibitem{Chen-Lih-Wu}
Bor~Liang Chen, Ko-Wei Lih, and Pou-Lin Wu.
\newblock Equitable coloring and the maximum degree.
\newblock {\em European J. Combin.}, 15(5):443--447, 1994.

\bibitem{Hajnal-Szemeredi}
A.~Hajnal and E.~Szemer\'{e}di.
\newblock Proof of a conjecture of {P}. {E}rd{\H o}s.
\newblock In {\em Combinatorial theory and its applications, {I}-{III} ({P}roc. {C}olloq., {B}alatonf\"{u}red, 1969)}, volume~4 of {\em Colloq. Math. Soc. J\'{a}nos Bolyai}, pages 601--623. North-Holland, Amsterdam-London, 1970.

\bibitem{Karpov}
Dmitri Karpov.
\newblock An upper bound on the number of edges in an almost planar bipartite graph.
\newblock {\em J. Math. Sci.}, 196:737--746, 2014.

\bibitem{Kostochka-Nakprasit2}
A.~V. Kostochka and K.~Nakprasit.
\newblock Equitable colurings of {$d$}-degenerate graphs.
\newblock {\em Combin. Probab. Comput.}, 12(1):53--60, 2003.

\bibitem{Kostochka-Nakprasit1}
A.~V. Kostochka and K.~Nakprasit.
\newblock On equitable {$\Delta$}-coloring of graphs with low average degree.
\newblock {\em Theoret. Comput. Sci.}, 349(1):82--91, 2005.

\bibitem{Kostochka-Lin-Xiang}
Alexandr Kostochka, Dou Lin, and Zimu Xiang.
\newblock Equitable coloring of planar graphs with maximum degree at least 8.
\newblock 2023.
\newblock \href {https://doi.org/10.48550/arXiv.2305.12041} {\path{doi:10.48550/arXiv.2305.12041}}.

\bibitem{Lih-survey}
Ko-Wei Lih.
\newblock {\em Equitable Coloring of Graphs}, pages 1199--1248.
\newblock Springer New York, 2013.

\bibitem{Lih-Wu}
Ko-Wei Lih and Pou-Lin Wu.
\newblock On equitable coloring of bipartite graphs.
\newblock {\em Discrete Math.}, 151:155--160, 1996.
\newblock Graph theory and combinatorics (Manila, 1991).

\bibitem{Meyer}
Walter Meyer.
\newblock Equitable coloring.
\newblock {\em Amer. Math. Monthly}, 80:920--922, 1973.

\bibitem{Nakprasit}
Kittikorn Nakprasit.
\newblock Equitable colorings of planar graphs with maximum degree at least nine.
\newblock {\em Discrete Math.}, 312(5):1019--1024, 2012.

\bibitem{Tian-Zhang}
Jingjing Tian and Xin Zhang.
\newblock Pseudo-outerplanar graphs and chromatic conjectures.
\newblock {\em Ars Combin.}, 114:353--361, 2014.

\bibitem{Yap-Zhang1}
Hian~Poh Yap and Yi~Zhang.
\newblock The equitable {$\Delta$}-colouring conjecture holds for outerplanar graphs.
\newblock {\em Bull. Inst. Math. Acad. Sinica}, 25(2):143--149, 1997.

\bibitem{Zhang}
Xin Zhang.
\newblock On equitable colorings of sparse graphs.
\newblock {\em Bull. Malays. Math. Sci. Soc.}, 39(1):S257--S268, 2016.

\bibitem{Yap-Zhang2}
Yi~Zhang and Hian-Poh Yap.
\newblock Equitable colorings of planar graphs.
\newblock {\em J. Combin. Math. Combin. Comput.}, 27:97--105, 1998.

\end{thebibliography}
}

\end{document}